\documentclass{amsart}

\usepackage{amsmath, amsthm, amssymb}
\usepackage{graphicx}
\usepackage{tikz, tikz-cd}
\usetikzlibrary{arrows, matrix}
\usepackage{ytableau}
\usepackage{subfigure}
\usepackage{hyperref}

\newcommand{\sg}{\sigma}

\def\inv{\operatorname{inv}}

\newcommand{\qint}[1] { [#1]_{q} }
\newcommand{\qtint}[1] { [#1]_{q,t} }

\newcommand{\qstir}[3] {S_{#1, #2}(#3)}

\def\set{\operatorname{set}}

\newcommand{\osp}[2] {\mathcal{OP}_{#1, #2}}
\def\bl{\operatorname{bl}}

\def\S{\mathfrak{S}}

\def\multiset#1#2{\ensuremath{\left(\kern-.3em\left(\genfrac{}{}{0pt}{}{#1}{#2}\right)\kern-.3em\right)}}

\def\syt{\operatorname{SYT}}

\def\area{\operatorname{area}}

\newcommand{\pf}{\mathcal{PF}}

\newcommand\tes[3]{\operatorname{Tes}_{#1}(#2,#3)}
\newcommand\tesset[1]{\mathcal{T}_{#1}}
\newcommand\ptesset[1]{\mathcal{PT}_{#1}}

\def\wt {\operatorname{wt}}
\def\pos{\operatorname{entries}^{+}}
\def\posrows{\operatorname{rows}^{+}}
\def\neg{\operatorname{entries}^{-}}
\def\negrows{\operatorname{rows}^{-}}

\def\hilb{\operatorname{Hilb}}
\def\vhilb{\widetilde{\operatorname{Hilb}}}
\def\sort{\operatorname{sort}}
\def\hooks{\alpha}
\def\hooksvec{\operatorname{hooks}}
\def\nonzero{\operatorname{nonzero}}
\def\entries{\operatorname{entries}}
\def\rows{\operatorname{rows}}

\def\el{\ell}

\def\arm{a}
\def\leg{\el}

\newcommand{\tail}{\operatorname{tail}}
\def\target{\operatorname{target}}

\def\car{\operatorname{car}}
\def\spot{\operatorname{spot}}
\def\considerate{\operatorname{cons}}
\def\consideratepf{\mathcal{CPF}}

\newtheorem{lemma}{Lemma}[section]
\newtheorem{prop}{Proposition}[section]
\newtheorem{cor}{Corollary}[section]
\newtheorem{thm}{Theorem}[section]

\author{Andrew Timothy Wilson }
\title{A weighted sum over generalized Tesler matrices}
\address{Department of Mathematics, University of Pennsylvania, 209 South 33rd Street, Philadelphia, PA 19104-6395}
\thanks{The author was supported by an NDSEG Fellowship and an NSF Postdoctoral Fellowship. Much of this work was done while the author was at UC San Diego. An extended abstract of this material appeared in the proceedings of FPSAC 2015. Finally, the author would like to thank Jim Haglund, Alejandro Morales, Jeff Remmel, and Brendon Rhoades for useful discussions.}
\email{\texttt{andwils@math.upenn.edu}}
\keywords{Tesler matrices, Macdonald polynomials, Shuffle Conjecture, ordered set partitions, parking functions}

\begin{document}
\maketitle

\begin{abstract}
We generalize previous definitions of Tesler matrices to allow negative matrix entries and negative hook sums. Our main result is an algebraic interpretation of a certain weighted sum over these matrices, which we call the Tesler function. Our interpretation uses a new class of symmetric function specializations which are defined by their values on Macdonald polynomials. As a result of this interpretation, we obtain a Tesler function expression for the Hall inner product $\langle \Delta_f e_n, p_{1^{n}}\rangle$, where $\Delta_f$ is the delta operator introduced by Bergeron, Garsia, Haiman, and Tesler. We also provide simple formulas for various special cases of Tesler functions which involve $q,t$-binomial coefficients, ordered set partitions, and parking functions. These formulas prove two cases of the recent Delta Conjecture posed by Haglund, Remmel, and the author.
\end{abstract}

\tableofcontents

\section{Introduction}
\label{sec:intro}
We say that a matrix $U \in M_{n}(\mathbb{Z})$ is a \emph{Tesler matrix} if 
\begin{enumerate}
\item $U$ is upper triangular,
\item \label{no-zero-rows} $U$ has no zero rows, and
\item \label{same-sign} every row of $U$ is entirely non-negative or entirely non-positive.
\end{enumerate}
Given a Tesler matrix $U$, $\hooksvec(U)$ is the vector whose $i$th entry is defined by
\begin{align*}
\hooksvec_i(U) &= (U_{i,i} + \ldots + U_{i,n}) - (U_{1,i} + \ldots + U_{i-1,i}) .
\end{align*}
We will sometimes call this the $i$th \emph{hook sum} of $U$. For example, the matrix
\begin{align*}
U = 
\left[ \begin{array}{rrrr}
0 & 1 & 0 & 2 \\
0 & -1 & -1 & 0 \\
0 & 0 & 1 & 0 \\
0 & 0 & 0 & 1
\end{array} \right]
\end{align*}
is a Tesler matrix with $\hooksvec(U) = (3,-3,2,-1)$. For $\hooks \in \mathbb{Z}^n$, we denote the set of all Tesler matrices with $\hooksvec(U) = \hooks$ by $\tesset{\hooks}$.

The cases $\hooks  = (1, 1, \ldots, 1)$ and $\hooks = (1, m, \ldots, m)$ for any positive integer $m$ are studied in \cite{tesler-hilbert}, where they are used to give an expression for the Hilbert series of the (generalized) module of diagonal harmonics. More values of $\hooks$ have appeared in the study of Hall-Littlewood polynomials \cite{tesler-lots}, Macdonald polynomial operators \cite{tesler-constant}, and flow polytopes \cite{tesler-polytope}. 

We will sometimes refer to a matrix that satisfies condition (2) above as \emph{essential} and a matrix that satisfies condition (3) as \emph{signed}. Since previous work on Tesler matrices primarily addresses matrices with positive hook sums, and conditions (2) and (3) are trivial in that setting, our definition generalizes previous definitions of Tesler matrices. 

Previous work on Tesler matrices for particular $\hooks$ has shown that it is unlikely that the number of Tesler matrices $|\tesset{\hooks}|$ has a simple formula. Perhaps the best way of counting Tesler matrices so far is to use the fact that they are integer points in the Tesler polytope \cite{tesler-polytope}; this gives us a concrete reason to try to extend the Tesler polytope to our setting.

Instead of attempting to count Tesler matrices, we will consider the \emph{weight} of an $n \times n$ Tesler matrix $U$
\begin{align*}
\wt(U;q,t) &= (-1)^{\pos(U) - \posrows(U)} M^{\nonzero(U) - n} \prod_{U_{i,j} \neq 0} \qtint{U_{i,j}} 
\end{align*}
where $M = (1-q)(1-t)$, $\pos(U)$ is the number of positive entries in $U$, $\posrows(U)$ is the number of rows of $U$ whose nonzero entries are all positive, $\nonzero(U)$ is the number of nonzero entries of $U$, and $\qtint{k} = \frac{q^k - t^k}{q-t}$, the usual $q,t$-analogue of an integer $k$. Since $U$ is essential, the exponent of $M$ is nonnegative and $\wt(U;q,t) \in \mathbb{Z}[q,t,1/q,1/t]$. When $U$ has no negative entries (which must be the case if each $\hooks_i$ is positive), this weight function is equal to the weight function defined in \cite{tesler-hilbert}. It is also worth noticing that the weight of a Tesler matrix is independent of $\hooks$. We define the \emph{Tesler function with hook sums $\hooks$} to be
\begin{align*}
\tes{\hooks}{q}{t} &= \sum_{U \in \tesset{\hooks}} \wt(U;q,t) \in \mathbb{Z}[q,t,1/q,1/t].
\end{align*}
In \cite{tesler-hilbert}, Haglund showed that $\tes{1^n}{q}{t}$ is equal to the Hilbert series of the module of diagonal harmonics, which can also be written in terms of Macdonald polynomial operators as $\langle \nabla e_n, p_{1^{n}}\rangle$ or $\langle \Delta_{e_n} e_n, p_{1^{n}}\rangle$. \cite{tesler-constant} contains an algebraic interpretation for $\tes{\hooks}{q}{t}$ for any $\hooks$ with only positive entries. We summarize these results, along with the necessary notation, in Section \ref{sec:background}.

In Section \ref{sec:hilbert}, we develop an algebraic interpretation for $\tes{\hooks}{q}{t}$ for any $\hooks \in \mathbb{Z}^n$ in terms of new symmetric function specializations which we call \emph{virtual Hilbert series}. Our interpretation is equivalent to the interpretation in \cite{tesler-constant} for positive hook sums; in this sense, the definition of Tesler matrices that we have used here is the natural extension of previous definitions. These specializations generalize the map that sends a symmetric function $f$ that is homogeneous of degree $n$ to its inner product with $p_{1^n}$. In the case that $f$ is the Frobenius image of an $\S_n$-module, this inner product extracts the module's Hilbert series. With this in mind, for \emph{any} symmetric function $f$ that is homogenous of degree $n$, we will often use the notation
\begin{align*}
\hilb f &= \langle f , p_{1^{n}}\rangle.
\end{align*}

In Section \ref{sec:harmonics}, we show that certain sums of virtual Hilbert series appear in the study of diagonal harmonics, especially in connection with the Macdonald polynomial operators $\Delta_f$ and $\Delta^{\prime}_f$. We use the algebraic interpretation of Tesler functions from Section \ref{sec:hilbert} to produce a number of new results about these operators. In particular, we show that $\hilb \Delta^{\prime}_f e_n$ and $\hilb \Delta_f e_n$ are always polynomials in $q$ and $t$ for any symmetric function $f$ and we derive a simple expression for $\hilb \Delta_{e_1} e_n$. 

In the remainder of the paper, we prove results about special cases of Tesler functions. Section \ref{sec:t=0} addresses the case where $\hooks \in \{0,1\}^n$ and $t=0$. In particular, we extend an involution of Levande \cite{levande} from the $\hooks = 1^n$ to case to show that
\begin{align*}
\tes{\hooks}{q}{0} &= \prod_{i=1}^{n} \qint{\hooks_1+\hooks_2+\ldots+\hooks_i}
\end{align*}
via a connection to ordered set partitions. This completes a case of the Delta Conjecture, described in \cite{delta-conjecture}. In Section \ref{sec:t=1} we address the $t=1$ case, providing a formula for $\tes{\hooks}{q}{1}$ for any $\hooks \in \mathbb{Z}^n$. Furthermore, we show that
\begin{align*}
\tes{\hooks}{1}{1} &= \hooks_1(\hooks_1 + n\hooks_2)(\hooks_1 + \hooks_2 + (n-1)\hooks_3) \ldots (\hooks_1 + \hooks_2 + \ldots + \hooks_{n-1} + 2\hooks_n) 
\end{align*}
for any $\hooks \in \mathbb{Z}^n$, extending a formula from \cite{tesler-lots}. These results exploit new connections between Tesler matrices, ordered set partitions, and parking functions. Finally, we discuss some potential future directions in Section \ref{sec:future}.

\section{Background}
\label{sec:background}

First, we fix some notation. We use $\Lambda$ to denote the algebra of symmetric functions over the base field $\mathbb{Q}(q,t)$. If we wish to refer to the elements of $\Lambda$ which are homogeneous of degree $n$, we will write $\Lambda^{(n)}$. Similarly, we will use $\underline{\Lambda}$ and $\underline{\Lambda}^{(n)}$ to refer to the algebra of symmetric Laurent polynomials. Occasionally, we will use ${\mathbb{Z}[q,t]}$ as a subscript to refer to the subalgebra of symmetric functions or symmetric Laurent polynomials consisting of the functions with coefficients in $\mathbb{Z}[q,t]$. 

There are several important classical bases for the space of symmetric functions (viewed as a vector space): the monomial symmetric functions $\{m_{\lambda}\}$, the elementary symmetric functions $\{e_{\lambda}\}$, the homogeneous symmetric functions $\{h_{\lambda}\}$, the power symmetric functions $\{p_{\lambda}\}$, and the Schur functions $\{s_{\lambda}\}$. The (modified) Macdonald polynomials $\{\widetilde{H}_{\lambda}\}$ are also a basis for this space, and generalize many of the important properties of the classical bases. We refer the reader to \cite{ec2, macdonald} for more material on symmetric functions and Macdonald polynomials. The only basis we will use for the symmetric Laurent polynomials is the monomial basis $\{m_{\rho}\}$, defined as the sum of all monomials whose exponents, when arranged in weakly decreasing order, equal the finite, weakly decreasing vector of nonzero integers $\rho$. We will refer to a finite vector of weakly decreasing nonzero integers as a \emph{Laurent partition}. Since each partition is also a Laurent partition, our choice to use $m$'s for both types of monomial bases is justified.

When studying Macdonald polynomials, it is quite useful to have the following notation. Given a partition $\mu \vdash n$ and a cell $c$ in the Young diagram of $\mu$ (drawn in French notation) we set $\arm(c), \arm^{\prime}(c)$, $\el(c)$, and $\el^{\prime}(c)$ to be the number of cells in $\mu$ that are strictly to the right of, to the left of, above, and below $c$ in $\mu$, respectively. In Figure \ref{fig:arm-leg}, we compute these values for a particular example.

\begin{figure}
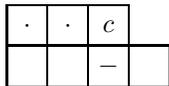

\begin{displaymath}
\begin{ytableau}
\cdot 	& \cdot	& c  \\
\phantom{\cdot}	&  \phantom{\cdot} 	&  -	&  \phantom{\cdot}
\end{ytableau}
\end{displaymath}
\caption{This is the Young diagram (in French notation) of the partition $(4,3)$. The cell $c$ has $\arm(c) = 0$, $a^{\prime}(c) = 2$ (represented by dots), $\el(c) = 0$, and $\el^{\prime}(c) = 1$ (represented by dashes).}
\label{fig:arm-leg}
\end{figure}

We set 
\begin{align*}
T_{\mu} &= \prod_{c \in \mu} q^{a^{\prime}(c)} t^{\el^{\prime}(c)} \ \ \ \ \ \ \ \ \ \ \ \ B_{\mu} = \sum_{c \in \mu} q^{a^{\prime}(c)} t^{\el^{\prime}(c)} 
\end{align*} 
We will use this notation to define a number of operators on $\Lambda^{(n)}$. Each of the operators is defined by its action on the Macdonald polynomial basis. First, we define
\begin{align*}
\nabla \widetilde{H}_{\mu} = T_{\mu} \widetilde{H}_{\mu} .
\end{align*}
The $\nabla$ operator has become quite famous due to its connection with the module of diagonal harmonics. Specifically, in \cite{delta} Haiman proved that the Frobenius image of the character of the module of diagonal harmonics of order $n$ is equal to $\nabla e_n$. For more on this module, see \cite{haglund-book}. The main result in \cite{tesler-hilbert} is that
\begin{align*}
\hilb \nabla e_n &= \tes{1^n}{q}{t}.
\end{align*}

Given any symmetric Laurent polynomial $f$, we define two more operators on $\Lambda^{(n)}$ by
\begin{align*}
\Delta_f  \widetilde{H}_{\mu} &= f[B_{\mu}] \widetilde{H}_{\mu} \ \ \ \ \ \ \ \ \ \ \ \ 
\Delta^{\prime}_f  \widetilde{H}_{\mu} = f[B_{\mu}-1] \widetilde{H}_{\mu}.
\end{align*}
Here, we have used the notation that, for a symmetric Laurent polynomial $f$ and a sum $S = s_1 + \ldots + s_k$ of monic monomials, $f[S]$ is equal to the specialization of $f$ at $x_1=s_1, \ldots, x_k=s_k$, where the remaining variables are set equal to zero. Note that both operators are linear in their subscripts. The operator $\Delta_f$, at least in the case where $f$ is a symmetric function, appears often in the study of diagonal harmonics \cite{haglund-book}. Furthermore, it is clear that $\Delta_{e_n} = \nabla$ when applied to $\Lambda^{(n)}$, so $\nabla e_n = \Delta_{e_n} e_n$. We allow for $f$ to be a symmetric Laurent polynomial because it provides a way to obtain negative powers of $\nabla$ in terms of our operators via the identity
\begin{align}
\nabla^{-1} &= \Delta_{m_{(-1)^n}}
\end{align}
on $\Lambda^{(n)}$.

Although the operator $\Delta^{\prime}_f$ is not as common, it can be connected to $\Delta_{f}$ via the identity
\begin{align}
\label{delta-prime}
\Delta_{m_{\rho}} = \Delta^{\prime}_{m_{\rho}} +  \sum_{\xi} \Delta^{\prime}_{m_{\xi}}
\end{align}
where the sum is over all $\xi$ which can be obtained by removing one part from $\rho$. In particular,
\begin{align}
\label{delta-prime-e}
\Delta_{e_{k}} &= \Delta^{\prime}_{e_{k}} + \Delta^{\prime}_{e_{k-1}}
\end{align} 
which, combined with the fact that $\Delta^{\prime}_{f} g = 0$ if the degree of $f$ is greater than or equal to the degree of $g$, implies $\nabla e_n = \Delta^{\prime}_{e_{n-1}} e_n$. 

We will also make use of the Pieri and skew Pieri coefficients of $\widetilde{H}_{\mu}$. We define the skewing operator on $\Lambda$ by insisting that
\begin{align*}
\langle f^{\perp} g, h \rangle &= \langle g, fh \rangle
\end{align*}
for any symmetric functions $f$, $g$, and $h$. Here and in the sequel, we use the Hall inner product on symmetric functions. Then the skew Pieri coefficients $c_{\mu, \nu}$ are defined by 
\begin{align*}
e_{1}^{\perp} \widetilde{H}_{\mu} &= \sum_{\nu \rightarrow \mu} c_{\mu, \nu} \widetilde{H}_{\nu} 
\end{align*}
where the sum is over all partitions $\nu$ that can be obtained by removing a single cell from $\mu$. In \cite{tesler-constant}, the authors use a constant term algorithm to provide a formula for $\tes{\hooks}{q}{t}$ for any vector $\hooks$ of positive integers in terms of the skewing operator. We will also use the Pieri coefficients\footnote{Some authors do not divide by  $M$ in the definition of the Pieri coefficients.} $d_{\mu, \nu}$, defined by
\begin{align*}
\frac{e_1}{M} \widetilde{H}_{\nu} &= \sum_{\mu \leftarrow \nu} d_{\mu, \nu} \widetilde{H}_{\mu} 
\end{align*}
where the sum is over all $\mu$ that can be created by adding a cell to $\nu$.

Finally, we will employ the following standard notation for $q,t$- and $q$-analogs of integers:
\begin{align*}
\qtint{n} &= \frac{q^n - t^n}{q-t} \\
\qint{n} &= [n]_{q,1} = \frac{q^n - 1}{q-1} .
\end{align*}
Note that $\qtint{n} \in \mathbb{N}[q,t]$ if $n \geq 0$ and $\qtint{n} \in \mathbb{Z}[1/q,1/t]$ if $n \leq 0$. This implies $\tes{\hooks}{q}{t} \in \mathbb{Z}[q,t,1/q,1/t]$ for any $\hooks \in \mathbb{Z}^n$.

\section{Virtual Hilbert series}
\label{sec:hilbert}

In this section, we use new symmetric function specializations to derive an algebraic interpretation for $\tes{\hooks}{q}{t}$ for any vector of integers $\hooks$. Our interpretation generalizes the formulas in \cite{tesler-hilbert, tesler-constant}. 

\subsection{Definitions and connections to Tesler functions}
\label{ssec:results}
Given any $\hooks \in \mathbb{Z}^{n-1}$ and $\mu \vdash n$, we make the following recursive definition.
\begin{align*}
F^{\hooks}_{\mu} &= \sum_{\nu \rightarrow \mu} c_{\mu, \nu} (T_{\mu} / T_{\nu})^{\hooks_{1}} F^{(\hooks_2, \ldots, \hooks_{n-1})}_{\nu} \\
F^{()}_{(1)} &= 1
\end{align*}
It is worth noting that $F^{0^{n-1}}_{\mu} = \hilb \widetilde{H}_{\mu}$, the Hilbert series of the Garsia-Haiman module associated with $\widetilde{H}_\mu$,  which is sometimes denoted $F_{\mu}$. As a result, $F^{\hooks}_{\mu}$ can be thought of as a modification of this Hilbert series. The famous $n!$ conjecture of Garsia and Haiman, proved in \cite{haiman-positivity}, is simply the statement that setting $q=t=1$ in $F_{\mu}$ yields $n!$. We list some open questions about the $F^{\hooks}_{\mu}$ below.
\begin{itemize}
\item Even though $c_{\mu,\nu}$ is generally in $\mathbb{Q}(q,t)$, computations in Sage suggest that $F^{\hooks}_{\mu}$ is always in $\mathbb{Z}[q,t,1/q,1/t]$ and $F^{\hooks}_{\mu} \in \mathbb{Z}[q,t]$ if $\hooks \in \mathbb{N}^n$. Is it true in general that $F^{\hooks}_{\mu} \in \mathbb{Z}[q,t,1/q,1/t]$ and that $\hooks \in \mathbb{N}^n$ implies $F^{\hooks}_{\mu} \in \mathbb{Z}[q,t]$?
\item For which $\hooks, \mu$ is $F^{\hooks}_{\mu} \in \mathbb{N}[q,t]$?
\item \cite{hhl} provides a combinatorial formula for $F_{\mu}$. Is there a similar formula for $F^{\hooks}_{\mu}$?
\end{itemize}

Now we define a map 
\begin{align*}
\vhilb_{\hooks} : \Lambda^{(n)} &\to \mathbb{Q}(q,t) \\
 \widetilde{H}_{\mu} &\mapsto F^{\hooks}_{\mu}
\end{align*}
We will sometimes refer to $\vhilb_{\hooks} f$ as the \emph{virtual Hilbert series} of $f$ with respect to $\hooks$. We can justify this terminology by noting that $F^{0^{n-1}}_{\mu} = F_{\mu}$ implies
\begin{align}
\vhilb_{0^{n-1}} &= \hilb 
\end{align}
on $\Lambda^{(n)}$. Furthermore, we have
\begin{align}
\vhilb_{k^{n-1}} &= \hilb \nabla^k   .
\end{align}
for any $k \in \mathbb{Z}$ on $\Lambda^{(n)}$. The following result gives an algebraic interpretation for $\tes{\hooks}{q}{t}$ for any $\hooks \in \mathbb{Z}^{n-1}$. We note that the right-hand side is equivalent to the right-hand side of I.9 in \cite{tesler-constant} if each entry of $\hooks$ is positive.

\begin{thm}
\label{thm:tes-hilb}
For any $\hooks \in \mathbb{Z}^{n-1}$, we have
\begin{align*}
\tes{\hooks}{q}{t} &= \frac{(-1)^{n-1}}{[n]_q [n]_t} \vhilb_{\hooks} p_{n} .
\end{align*}
\end{thm}

If we are willing to restrict our attention to vectors that begin with a 1, we can simplify the right-hand side of Theorem \ref{thm:tes-hilb} slightly. We also obtain a direct generalization of the results in \cite{tesler-hilbert}.

\begin{cor}
\label{cor:tes-hilb1}
For any $\hooks \in \mathbb{Z}^{n-1}$, we have
\begin{align*}
\tes{(1, \hooks)}{q}{t} &= \vhilb_{\hooks} e_{n} .
 \end{align*}
\end{cor}

\subsection{Proof of Theorem \ref{thm:tes-hilb}}
\label{ssec:proofs}

We will need the following lemmas in order to prove Theorem \ref{thm:tes-hilb}. For $\mu \leftarrow \nu$, we abbreviate $T_{\mu}/T_{\nu}$ by $T$. We also use an overline to indicate the operation of replacing $q$ by $1/q$ and $t$ by $1/t$. For example, $\overline{q+qt} = 1/q + 1/(qt)$.

\begin{lemma}
\label{lemma:perp-d}
For any partition $\nu$ and $k \in \mathbb{Z}$, we have
\begin{align}
\label{perp-d1}
\sum_{\mu \leftarrow \nu} d_{\mu, \nu} T^k &= \left\{
\begin{array}{lr}
(-1)^{k-1}e_{k-1}[MB_{\nu}-1]/M & k > 0 \\
1/M & k = 0  \\
\frac{(-1)^{-k}}{qt} \overline{e_{-k}[MB_{\nu}-1]/M} & k < 0 
\end{array} \right.
\end{align}
and, as a result,
\begin{align}
\label{perp-d2}
\sum_{\mu \leftarrow \nu} d_{\mu, \nu} (1-T) T^k &= \left\{
\begin{array}{lr}
(-1)^{k-1}e_{k} [MB_{\nu}] / M & k > 0 \\
0 & k=0 \\
\frac{(-1)^{-k}}{qt}  \overline{e_{-k} [MB_{\nu} ] / M } & k < 0 .
\end{array} \right.
\end{align}
\end{lemma}

\begin{proof}
The $k \geq 0$ case of \eqref{perp-d1} was first noticed by Zabrocki and proved in \cite{macdonald-pieri}. The $k \geq 0$ of \eqref{perp-d2} was shown to follow from \eqref{perp-d1} in \cite{tesler-hilbert}. We begin by proving the $k < 0$ case of \eqref{perp-d1}, which follows from the $k \geq 0$ case of \eqref{perp-d1} due to the following argument of Garsia (personal communication, 2015). 

First, we need to relate $d_{\mu, \nu}$ to $\overline{d_{\mu, \nu}}$. We will use the identity $\widetilde{H}_{\mu} = T_{\mu} \omega \overline{\widetilde{H}_{\mu}}$ \cite{macdonald}. By definition, we have
\begin{align}
\sum_{\mu \leftarrow \nu} d_{\mu, \nu} \widetilde{H}_{\mu} &= \frac{e_1}{M} \widetilde{H}_{\nu} \\
&= \frac{e_1}{M} T_{\nu} \omega \overline{\widetilde{H}_{\nu}}\\
&= \frac{T_{\nu}}{qt} \omega \left( \frac{e_1}{\overline{M}} \overline{\widetilde{H}_{\nu}} \right) \\
&= \frac{T_{\nu}}{qt} \omega \left( \sum_{\mu \leftarrow \nu} \overline{d_{\mu, \nu}} \overline{\widetilde{H}_{\mu}} \right) \\
&= \frac{T_{\nu}}{qt} \sum_{\mu \leftarrow \nu} \overline{d_{\mu, \nu}} \omega \overline{\widetilde{H}_{\mu}} \\
&= \sum_{\mu \leftarrow \nu} \overline{d_{\mu, \nu}} \frac{1}{qtT} \widetilde{H}_{\mu}
\end{align}
which implies $d_{\mu, \nu} = \frac{1}{qtT} \overline{d_{\mu, \nu}}$. Now, for $k < 0$, we have
\begin{align}
\sum_{\mu, \nu} d_{\mu, \nu} T^k &= \frac{1}{qt} \sum_{\mu, \nu} \overline{d_{\mu, \nu}} T^{k-1} \\
&= \frac{1}{qt} \overline{ \sum_{\mu \leftarrow \nu} d_{\mu, \nu} T^{-k+1} } \\
&= \frac{(-1)^{-k}}{qt} \overline{ e_{-k}[MB_{\nu} - 1]/M } .
\end{align}
This proves \eqref{perp-d1}. The same plethystic computation used to derive Lemma 1 from (13) in \cite{tesler-hilbert} can be used to prove \eqref{perp-d2}.

\end{proof}

\begin{lemma}[Lemma 2 in \cite{tesler-hilbert}]
\label{lemma:qt-binom}
For any positive integer $k$,
\begin{align*}
(-1)^{k-1} e_k[M] / M &= \qtint{k} .
\end{align*}
\end{lemma}

Our proof of Theorem \ref{thm:tes-hilb} will closely follow the main proof in \cite{tesler-hilbert}. First, we note that \cite{macdonald}
\begin{align}
\frac{(-1)^{n-1}}{[n]_q [n]_t} p_n &= \sum_{\mu \vdash n} \frac{M \Pi_{\mu} }{w_{\mu}} \widetilde{H}_{\mu} 
\end{align}
where 
\begin{align*}
\Pi_{\mu} &= \prod_{\substack{ c \in \mu \\ c \neq (0,0)}} (1-q^{\arm^{\prime}(c)} t^{\leg^{\prime}(c)}) \\
w_{\mu} &= \prod_{c \in \mu} (q^{\arm(c)} - t^{\leg(c)+1}) (t^{\leg(c)} - q^{\arm(c)+1}) .
\end{align*}
Thus, the right-hand side of Theorem \ref{thm:tes-hilb} equals
\begin{align}
\vhilb_{\hooks} \left( \sum_{\mu \vdash n} \frac{M \Pi_{\mu} }{w_{\mu}} \widetilde{H}_{\mu} \right) 
 &= \sum_{\mu \vdash n} \frac{M \Pi_{\mu}}{w_{\mu}} \sum_{\nu \rightarrow \mu} c_{\mu, \nu} T^{\hooks_1} F_{\nu}^{(\hooks_2, \ldots, \hooks_{n-1})} \\
 \label{tes-hilb1}
 &= \sum_{\nu \vdash n-1} M F_{\nu}^{(\hooks_2, \ldots, \hooks_{n-1})} \sum_{\mu \leftarrow \nu} \frac{\Pi_{\mu}}{w_{\mu}} c_{\mu, \nu} T^{\hooks_1} 
\end{align}
where we have used the definition of $\vhilb_{\hooks}$ and switched the order of the sums. Using
\begin{align}
\Pi_{\mu} &= (1-T) \Pi_{\nu} \\
\frac{c_{\mu, \nu}}{w_{\mu}} &= \frac{d_{\mu, \nu}}{w_{\nu}} 
\end{align}
from the definition of $\Pi_{\mu}$ and from \cite{qtcatalan}, respectively,  (\ref{tes-hilb1}) equals
\begin{align}
\label{tes-hilb2}
\sum_{\nu \vdash n-1} \frac{M \Pi_{\nu}}{w_{\nu}} F_{\nu}^{(\hooks_2, \ldots, \hooks_{n-1})} \sum_{\mu \leftarrow \nu} d_{\mu, \nu} (1-T) T^{\hooks_1}.
\end{align}
Now we use Lemma \ref{lemma:perp-d} to simplify the inner sum. For the sake of compactness, let $b_{k} = b_{k}(\nu)$ equal the right-hand side of Lemma \eqref{perp-d2} that corresponds to $k \in \mathbb{Z}$ and $a_k = (-1)^{k-1} e_k[M]/M$. 
Then  (\ref{tes-hilb2}) equals
\begin{align}
\label{tes-hilb3}
\sum_{\nu \vdash n-1} \frac{M \Pi_{\nu}}{w_{\nu}} F_{\nu}^{(\hooks_2, \ldots, \hooks_{n-1})} b_{\hooks_1}.
\end{align}

We would like to iterate this argument. The only real difficulty comes from $b_{\hooks_1}$. In particular, we need to know how to simplify expressions of the form $\left. b_{k} \right|_{B_{\nu} \mapsto B_{\nu}+T}$. From \cite{tesler-hilbert}, for $k > 0$ we get
\begin{align}
b_k |_{B_{\nu} \mapsto B_{\nu} + T} &= b_k + T^k a_k - \sum_{j=1}^{k-1} MT^{k-j} a_{k-j} b_j .
\end{align}
From Lemmas \ref{lemma:perp-d} and \ref{lemma:qt-binom}, we have
\begin{align}
\label{neg-b}
\overline{b_k} &= -qt b_{-k} \\
\label{neg-a}
\overline{a_k} &= -qt a_k .
\end{align}
We also have $\overline{M} = \frac{M}{qt}$ by definition. Using these identities, for $k > 0$ we compute
\begin{align} 
b_{-k} |_{B_{\nu} \mapsto B_{\nu} + T} &= - \frac{\overline{b_k |_{B_{\nu} \mapsto B_{\nu} + T}}}{qt} \\
&= - \frac{\overline{b_k} + T^{-k} \overline{a_k} - \sum_{j=1}^{k-1} \overline{M} T^{j-k} \overline{a_{k-j}} \overline{b_j}}{qt}  \\
&= b_{-k} + T^{-k} a_k + \sum_{j=1}^{k-j} M T^{j-k} a_{k-j} b_{-j} .
\end{align}
%more would be nice here

Iterating this procedure $r$ times, we obtain an expression for the right-hand side of Theorem \ref{thm:tes-hilb} of the form
\begin{align}
\sum_{\nu \vdash n-r+1} \frac{M \Pi_{\nu}}{w_{\nu}} F_{\nu}^{(\hooks_{r+1}, \ldots, \hooks_{n-1})} A^{\hooks}_{r}
\end{align}
where $A^{\hooks}_r$ is some expression in the $a_k$'s and $b_k$'s. Moreover, we can compute $A^{\hooks}_{r+1}$ from $A^{\hooks}_r$ by the following recursive procedure.
\begin{enumerate}
\item Replace the $b_k$'s in $A^{\hooks}_r$ with 
\begin{align*} 
\begin{array}{ll} 
b_k + T^k a_k - \sum_{j=1}^{k-1} MT^{k-j} a_{k-j} b_j & \text{if }k > 0 \\
b_{k} + T^{k} a_{-k} + \sum_{j=1}^{-k-j} M T^{-k+j} a_{-k-j} b_{-j} & \text{if } k < 0 .
\end{array}  
\end{align*} 
\item Expand to form a Laurent polynomial in $T$, say $\sum_{j} \gamma_j T^j$.
\item Replace each $T^j$ with $b_{j+\hooks_{r+1}}$. 
\end{enumerate}
At $r=n$, we obtain
\begin{align}
\frac{(-1)^{n-1}}{[n]_q [n]_t} \vhilb_{\hooks} p_n &= \sum_{\nu \vdash 1} \frac{M \Pi_{\nu}}{w_{\nu}} A^{\hooks}_n = A^{\hooks}_n
\end{align}
since $\Pi_{(1)} = 1$ and $w_{(1)} = M$. Our next goal is to prove the following expression for $A^{\hooks}_n$ in terms of Tesler matrices

\begin{lemma}
\label{lemma:tesler-expression}
\begin{align*}
A^{\hooks}_n &= \sum_{U \in \tesset{\hooks}} (-1)^{\pos(A) - \posrows(A)} M^{\nonzero(A) - n}  \prod_{U_{i,i} \neq 0} b_{U_{i,i}} \prod_{U_{i,j} \neq 0: \, i \neq j} a_{U_{i,j}}
\end{align*}
where we have set $\nu = (1)$. 
\end{lemma}

\begin{proof}
We will prove this claim by induction on $n = \ell(\hooks)+1$.
%base case?
The crux of the induction step is noticing that there is a recursion on Tesler matrices. Given a Tesler matrix $U \in \tesset{(\hooks_1, \ldots, \hooks_{p-1})}$, we create a Tesler matrix $V \in \tesset{(\hooks_1, \ldots, \hooks_{p})}$ as follows:
\begin{enumerate}
\item For each row $i$ in $U$, we ``move'' some of the diagonal entry $U_{i,i}$ to the far right to create $v_{i,p}$. 
\item To create row $p$, we choose $v_{p,p}$ such that the $p$th hook sum of $V$ is $\hooks_p$. 
\end{enumerate}
For example, one way to send a matrix in $\tesset{(3,-3,2)}$ to a matrix in $\tesset{(3,-3,2,0)}$ is depicted below.
\begin{align*}
\left[ \begin{array}{rrr}
0 & 3 & 0 \\
0 & -1 & -1 \\
0 & 0 & 1 \\
\end{array} \right] \mapsto
\left[ \begin{array}{rrrr}
0 & 1 & 0 & 2 \\
0 & 0 & -1 & -1 \\
0 & 0 & 1 & 0 \\
0 & 0 & 0 & 1
\end{array} \right]
\end{align*}
Now we check that this recursion matches the recursive procedure for generating $A^{\hooks}_n$. We should think of each $b_k$ as representing some integer $k$ on the diagonal in $U$ and each $a_k$ as representing an off-diagonal entry $k$. Then Step 1 of the procedure for generating $A^{\hooks}_n$ corresponds to moving some part of each diagonal entry into the new rightmost column. The power of $T$ tracks the entries in this new rightmost column. Note that a new $M$ appears each time we increase the number of nonzero entries in the matrix and a new $-1$ appears each time we create a new positive entry. Step 3 of the procedure corresponds to choosing the new bottom right entry such that the final hook sum is correct. 
\end{proof}

Finally, we note that, for $\nu = (1)$, 
\begin{align}
a_{|k|} = b_k = \qtint{k} 
\end{align}
for any integer $k$ by Lemma \ref{lemma:qt-binom}. This fact, along with Lemma \ref{lemma:tesler-expression}, concludes the proof of Theorem \ref{thm:tes-hilb}. Corollary \ref{cor:tes-hilb1} follows by essentially the same argument except we use the expansion
\begin{align}
e_n &= \sum_{\mu \vdash n} \frac{M B_{\mu} \Pi_{\mu}  }{w_{\mu}} \widetilde{H}_{\mu} .
\end{align}
instead of the expansion for $p_n$.

\section{Applications to delta operators}
\label{sec:harmonics}

The right-hand sides of Theorem \ref{thm:tes-hilb} and Corollary \ref{cor:tes-hilb1} bear some similarity to symmetric function expressions popular in the study of diagonal harmonics. In this section, we explore these connections and use the connections to prove new results about the Macdonald polynomial operators $\Delta_f$ and $\Delta^{\prime}_f$.

\subsection{From virtual Hilbert series to delta operators}
\label{ssec:poly}

Recall that we have defined an operator $\Delta^{\prime}_{f}$ on $\Lambda^{(n)}$ by stating that it acts on the Macdonald polynomials by 
\begin{align*}
\Delta^{\prime}_{f} \widetilde{H}_{\mu} &= f[B_{\mu}-1] \widetilde{H}_{\mu} .
\end{align*}
Although we will not be able to describe every virtual Hilbert series in terms of this operator, we do have the following result involving symmetric sums of virtual Hilbert series. We let $\sort$ be the map that removes the zeros from $\hooks$ and then sorts the remaining entries in weakly decreasing order. 

\begin{thm}
\label{thm:delta-prime-hilb}
For any Laurent partition $\rho$ and positive integer $n$,
\begin{align*}
\hilb \Delta^{\prime}_{m_{\rho}} &= \sum_{\substack{\hooks \in \mathbb{Z}^{n-1} \\ \sort(\hooks) = \rho}} \vhilb_{\hooks} 
\end{align*}
as operators on $\Lambda^{(n)}$.
\end{thm}

\begin{proof}
We will show that these operators are equal by showing that their actions are equal on the Macdonald polynomial $\widetilde{H}_{\mu}$ for any $\mu \vdash n$. The right-hand side of the statement in the theorem equals
\begin{align}
\label{unrecurse}
\sum_{\substack{\hooks \in \mathbb{Z}^{n-1} \\ \sort(\hooks) = \rho}} F^{\hooks}_{\mu} \widetilde{H}_{\mu} .
\end{align}
Now we iterate through the recursive definition of $F^{\hooks}_{\mu}$ in order to obtain an alternative definition. Rather than just considering $\nu \rightarrow \mu$, we can consider all the saturated chains in Young's lattice from $\emptyset$ to $\mu$. These are in bijection with the standard Young tableaux of shape $\mu$, denoted $\syt(\mu)$. Let $c_S$ equal the product of all $c_{\mu, \nu}$'s that we encounter on the saturated chain from $\emptyset$ to $\mu$ associated with a given $S \in \syt(\mu)$. Given a cell $d$ in the Young diagram of $\mu$, let $S(d)$ denote the entry in cell $d$ in $S$. Then (\ref{unrecurse}) equals
\begin{align}
&\sum_{\substack{\hooks \in \mathbb{Z}^{n-1} \\ \sort(\hooks) = \rho}} \sum_{S \in \syt(\mu)} c_S \prod_{\substack{d \in \mu \\  d \neq (0,0) }} \left( q^{a^{\prime}(d)} t^{\el^{\prime}(d)} \right)^{\hooks_{n+1-S(d)}} \widetilde{H}_{\mu}  \\
=& \sum_{S \in \syt(\mu)} c_S \sum_{\substack{\hooks \in \mathbb{Z}^{n-1} \\ \sort(\hooks) = \rho}} \prod_{\substack{d \in \mu \\  d \neq (0,0) }} \left( q^{a^{\prime}(d)} t^{\el^{\prime}(d)} \right)^{\hooks_{n+1-S(d)}} \widetilde{H}_{\mu} \\
=& \sum_{S \in \syt(\mu)} c_S m_{\rho}[B_{\mu}-1] \widetilde{H}_{\mu} \\
=& F_{\mu} m_{\rho}[B_{\mu}-1] \widetilde{H}_{\mu} \\
=&  \hilb \Delta^{\prime}_{m_{\rho}} \widetilde{H}_{\mu}. \qedhere
\end{align}
\end{proof}

Applying Theorem \ref{thm:delta-prime-hilb} to the Tesler function expressions obtained in Section \ref{sec:hilbert}, we obtain the following identities.

\begin{cor}
\label{cor:delta-hilb}
\begin{align}
\label{delta-hilb-p}
\frac{(-1)^{n-1}}{[n]_q [n]_t} \hilb \Delta^{\prime}_{m_{\rho}}  p_{n} &= \sum_{\hooks: \sort(\hooks) = \rho} \tes{\hooks}{q}{t} \\
\label{delta-hilb-e}
\hilb \Delta^{\prime}_{m_{\rho}}  e_{n} &= \sum_{\hooks: \sort(\hooks) = \rho} \tes{(1, \hooks)}{q}{t} 
\end{align}
As a result, both left-hand sides are in $\mathbb{Z}[q,t,1/q,1/t]$. Furthermore, by the linearity in the subscript of $\Delta^{\prime}_{f}$ we have
\begin{align}
\label{contain1}
\frac{(-1)^{n-1}}{[n]_q [n]_t} \hilb \Delta^{\prime}_{f}  p_{n},\  \hilb \Delta^{\prime}_{f}  e_{n} \in \mathbb{Z}[q,t,1/q,1/t] \\
\label{contain2}
\frac{(-1)^{n-1}}{[n]_q [n]_t} \hilb \Delta^{\prime}_{g}  p_{n},\  \hilb \Delta^{\prime}_{g}  e_{n} \in \mathbb{Z}[q,t] 
\end{align}
for any $f \in \underline{\Lambda}_{\mathbb{Z}[q,t]}$, $g \in \Lambda_{\mathbb{Z}[q,t]}$. Finally, by (\ref{delta-prime}) we could replace $\Delta^{\prime}$ by $\Delta$ in (\ref{contain1}), (\ref{contain2}).
\end{cor}

Corollary \ref{cor:delta-hilb} can be thought of as a more concrete version of a special case of Theorem 1.3 in \cite{bght-positivity}, in which the authors showed that $\Delta_{f} \Lambda_{\mathbb{Z}[q,t]} \subseteq \Lambda_{\mathbb{Z}[q,t]}$ for any $f \in \Lambda_{\mathbb{Z}[q,t]}$. Corollary \ref{cor:delta-hilb} provides the first direct formulas for Hilbert series of expressions of this type.

It may be of interest to the reader to use Corollary \ref{cor:delta-hilb} in order to explicitly compute some $\hilb \Delta_{f} e_n$. Rather than state the exact analog of (\ref{delta-hilb-e}) for this case, we mention that the following process accomplishes this task.
\begin{enumerate}
\item Expand $f$ into variables $x_1, x_2, \ldots, x_{n-1}, 1$. 
\item Replace each monomial $x^{\hooks}$ in this expansion with  $\tes{(1,\hooks)}{q}{t}$.
\end{enumerate}
As an example, we compute $\hilb \Delta_{s_{3,2,1}} e_3$ using
\begin{align}
s_{3,2,1}(x_1, x_2, 1) &= x_1^3 x_2^2 + x_1^2 x_2^3 + x_1^3 x_2 + 2 x_1^2 x_2^2 + x_1 x_2^3 + x_1^2 x_2 + x_1 x_2^2 .
\end{align}
Replacing each monomial with its associated Tesler function, we get
\begin{align}
\hilb \Delta_{s_{3,2,1}} e_3 =& \tes{(1,3,2)}{q}{t} + \tes{(1,2,3)}{q}{t} + \tes{(1,3,1)}{q}{t}  \\
\nonumber
&+ 2\tes{(1,2,2)}{q}{t} + \tes{(1,1,3)}{q}{t} \\
\nonumber
&+ \tes{(1,2,1)}{q}{t} + \tes{(1,1,2)}{q}{t}. 
\end{align}
One concludes the calculation by computing each of the Tesler functions. We have not explored how this method compares to current methods for computing $\hilb \Delta_f e_n$ from a computational perspective.

\subsection{Positive formulas}
\label{ssec:pos}
In this subsection, we use Corollary \ref{cor:delta-hilb} to obtain formulas for $\hilb \Delta_{e_1} e_n$, $\hilb \frac{(-1)^{n-1}}{[n]_q [n]_t} \hilb \Delta_{e_2} p_n$, and $\hilb \Delta_{m_{-1}} e_n$. Each formula shows that the Hilbert series of the given symmetric function is positive with respect to some set of variables.

\begin{cor}
\label{cor:e1}
\begin{align*}
\hilb \Delta_{e_1} e_n &= \sum_{k=1}^{n} \binom{n}{k} \qtint{k} \\
\frac{(-1)^{n-1}}{[n]_q [n]_t} \hilb \Delta_{e_2}  p_n &=  \sum_{k=1}^{n-1} \binom{n-1}{k} \qtint{k} .
\end{align*}
In particular, the left-hand sides of both statements are in $\mathbb{N}[q,t]$.
\end{cor}

\begin{cor}
\label{cor:m-1}
\begin{align*}
\hilb \Delta_{m_{-1}} e_n &= \left(1 - \frac{1}{qt}\right)^{n-1} .
\end{align*}
As a result, the left-hand side is in $\mathbb{N}\left[ -\frac{1}{qt} \right]$.
\end{cor}

In general, none of these symmetric functions are currently associated with modules, which means that direct formulas such as these are the only way to give positivity results at this point. In \cite{delta-conjecture}, Haglund, Remmel, and the author use a reciprocity identity to obtain the full Schur expansion of $\Delta_{e_1} e_n$, implying the $e_1$ statement in Corollary \ref{cor:e1}. Before we can prove Corollaries \ref{cor:e1} and \ref{cor:m-1}, we need the following lemma.

\begin{lemma}
\label{lemma:remove-1}
For any $\hooks \in \mathbb{Z}^{n}$,
\begin{align*}
\tes{(1,\hooks)}{q}{t} &= \tes{\hooks}{q}{t} + \sum_{i=1}^{n} \tes{(\hooks_1, \ldots, \hooks_{i-1}, \hooks_i + 1, \hooks_{i+1}, \ldots, \hooks_n)}{q}{t}.
\end{align*}
\end{lemma}

\begin{proof}
Consider a Tesler matrix $U$  with hook sums $(1, \hooks)$. Its first row must consist of a single nonzero entry, which must be equal to 1. Say this entry occurs in column $j$, i.e.\ $U_{1,j} = 1$. If $j=1$, removing the first row of $U$ produces a Tesler matrix  with hook sums $\hooks$, and this process produces a new matrix with hook sums $(\hooks_1, \ldots, \hooks_n)$. If $j>1$, we produce a Tesler matrix with hook sums $(\hooks_1, \ldots, \hooks_{j-2},\hooks_{j-1}+1,\hooks_{j}, \ldots, \hooks_n)$. Finally, we note that removing the first row does not change the weight of such a Tesler matrix.
\end{proof}

\begin{proof}[Proof of Corollary \ref{cor:e1}]
By (\ref{delta-prime-e}), the left-hand side of the statement involving $e_1$ in Corollary \ref{cor:e1} is equal to 
\begin{align}
\label{tes1}
\tes{(1,0^{n-1})}{q}{t}  + \sum_{i=0}^{n-2} \tes{(1,0^{i}, 1,0^{n-i-2})}{q}{t} .
\end{align}
In order to simplify this expression, we use Lemma \ref{lemma:remove-1} along with the fact that $\tes{\hooks}{q}{t} = 0$ if $\hooks_1 = 0$. As a result, (\ref{tes1}) equals
\begin{align}
&= \tes{(1,0^{n-1})}{q}{t} + \sum_{i=0}^{n-2} \tes{(1,1,0^i)}{q}{t} \\
&= \hilb \Delta_{e_1} e_{n-1} + \tes{(1,1,0^{n-2})}{q}{t} .
\end{align}
Applying Lemma \ref{lemma:remove-1} again, we get
\begin{align}
\tes{(1,1,0^{n-2})}{q}{t} =& \tes{(2,0^{n-2})}{q}{t} + \tes{(1,0^{n-2})}{q}{t} \\
\nonumber
&+ \sum_{i=0}^{n-3} \tes{(1,0^i, 1,0^{n-i-3})}{q}{t} \\
=& \tes{(2,0^{n-2})}{q}{t} + \hilb\Delta_{e_1} e_{n-1}  .
\end{align}
Therefore
\begin{align}
\label{e1-recursion}
\hilb\Delta_{e_1} e_{n} &= 2\hilb\Delta_{e_1} e_{n-1}  + \tes{(2,0^{n-2})}{q}{t} .
\end{align}

We claim that 
\begin{align}
\label{2-zeros}
\tes{(2,0^k)}{q}{t} &= \sum_{i=1}^{k+2} \left( \binom{k+1}{i-1} - \binom{k+1}{i} \right) \qtint{i} .
\end{align}
If we can prove this, induction on (\ref{e1-recursion}) implies
\begin{align}
\hilb \Delta_{e_1} e_n &= 2 \sum_{k=1}^{n-1} \binom{n-1}{k} \qtint{k}  + \sum_{k=1}^{n} \left( \binom{n-1}{k-1} - \binom{n-1}{k} \right) \qtint{k} \\
&= \sum_{k=1}^{n} \left(2 \binom{n-1}{k} + \binom{n-1}{k-1} - \binom{n-1}{k} \right) \qtint{k}  \\
&= \sum_{k=1}^{n} \binom{n}{k} \qtint{k} 
\end{align}
concluding the proof. We consider what happens when we remove the first row of a Tesler matrix with hook sums $(2,0^{n-2})$. The only way such a matrix can avoid having a zero row is if its first row contains a 2 in column 2 or a 1 in column 2 and a 1 in some other column. By removing the first row and using induction, we get
\begin{align}
\tes{(2,0^k)}{q}{t} &= \qtint{2} \tes{(2,0^{k-1})}{q}{t} - M \hilb \Delta_{e_1} e_k .
\end{align}
We can use induction to write this as 
\begin{align}
\tes{(2,0^k)}{q}{t} &= \qtint{2} \sum_{i=1}^{k+1} \left( \binom{k}{i-1} - \binom{k}{i} \right) \qtint{i} - M \sum_{i=1}^{k} \binom{k}{i} \qtint{i} \\
&= \sum_{i=1}^{k+1} \left( (q+t) \left(\binom{k}{i-1} - \binom{k}{i}\right) - (1-q)(1-t) \binom{k}{i} \right) \qtint{i} \\
&= \sum_{i=1}^{k+1} \left( (q+t) \binom{k}{i-1} - (1+qt) \binom{k}{i} \right) \qtint{i} .
\end{align}
Now all that remains to show is that, for any $a, b \geq 0$, the coefficient of $q^a t^b$ in the previous statement is $\binom{k+1}{a+b} - \binom{k+1}{a+b+1}$. 
This coefficient equals
\begin{align}
-&\binom{k}{a+b+1} + 2\binom{k}{a+b-1} - \binom{k}{a+b-1} \\
=& \binom{k}{a+b-1} - \binom{k}{a+b+1} \\
=& \left(\binom{k}{a+b-1} + \binom{k}{a+b}\right) - \left(\binom{k}{a+b} + \binom{k}{a+b+1} \right) \\
=& \binom{k+1}{a+b} - \binom{k+1}{a+b+1} .
\end{align}
We omit the proof of the second statement in the corollary, as it follows directly from the argument above and Theorem \ref{thm:tes-hilb}.
\end{proof}

To prove Corollary \ref{cor:m-1}, we will need another lemma about Tesler functions.

\begin{lemma}
\label{lemma:negative-hooks}
Given $\hooks \in \mathbb{Z}^{n}$, let $-\hooks = (-\hooks_1, \ldots, -\hooks_n)$. Then
\begin{align*}
\tes{-\hooks}{q}{t} &= \left(-\frac{1}{qt} \right)^n \tes{\hooks}{1/q}{1/t}.
\end{align*}
\end{lemma}

\begin{proof}
Given $U \in \tesset{\hooks}$, consider the matrix $-U$. Clearly $-U \in \tesset{-\hooks}$. Furthermore, $\wt(-U;q,t)$
\begin{align}
=&  (-1)^{\pos(-U) - \posrows(-U)}  M^{\nonzero(-U)-n} \prod_{-U_{i,j} \neq 0} \qtint{-U_{i,j}} \\
\label{eq-neg}
=& (-1)^{\neg(U) - \negrows(U)} (qt\overline{M})^{\nonzero(U)-n} \\
\nonumber
&\times \prod_{U_{i,j} \neq 0} \left(-\frac{1}{qt} \right) [U_{i,j}]_{1/q,1/t}.
\end{align}
We check that
\begin{align}
(-1)^{\neg(U) - \negrows(U) + \nonzero(U)} &= (-1)^{\pos(U) - \posrows(U) + n} \\
qt^{\nonzero(U) - n - \nonzero(U)} &= qt^{-n}
\end{align}
thus \ref{eq-neg} equals $(-qt)^{-n} \tes{\hooks}{1/q}{1/t}$.
\end{proof}

\begin{proof}[Proof of Corollary \ref{cor:m-1}]
We write
\begin{align}
\Delta_{m_{-1}} e_n &= \tes{(1,0^{n-1})}{q}{t} + \sum_{i=0}^{n-2} \tes{(1,0^i,-1,0^{n-i-2})}{q}{t} \\
&= \tes{(1,0^{n-1})}{q}{t} + \sum_{i=0}^{n-2} \tes{(1,-1,0^i)}{q}{t} \\
&= \Delta_{m_{-1}} e_{n-1} + \tes{(1,-1,0^{n-2})}{q}{t}
\end{align}
by induction. It is enough to show that
\begin{align}
\tes{(1,-1,0^k)}{q}{t} = -\frac{1}{qt} \left(1-\frac{1}{qt}\right)^k
\end{align}
for $k \leq n-2$. To accomplish this, we use Lemmas \ref{lemma:remove-1} and \ref{lemma:negative-hooks} to write
\begin{align}
\tes{(1,-1,0^k)}{q}{t} =& \tes{(-1,0^k)}{q}{t} + \sum_{i=0}^{k-1} \tes{(-1,0^i, 1, 0^{k-i-1})}{q}{t}\\
=& \left(-\frac{1}{qt}\right)^{k+1} \tes{(1,0^k)}{1/q}{1/t} \\
\nonumber
&+ \left(-\frac{1}{qt}\right)^{k+1} \sum_{i=0}^{k-1} \tes{(1,0^i,-1,0^{k-i-1})}{1/q}{1/t} \\
=& \left(-\frac{1}{qt}\right)^{k+1} \overline{\left( \hilb \Delta_{m_{-1}} e_{k+1}  \right)} \\
=& \left(-\frac{1}{qt}\right)^{k+1} (1-qt)^k \\ 
=& \left(-\frac{1}{qt}\right) (1-1/qt)^k 
\end{align}
by induction on $k$.
\end{proof}

\section{Combinatorics at $t=0$}
\label{sec:t=0}
In this subsection, we show how to relate $\left. \hilb \Delta^{\prime}_{e_k} e_n \right|_{t=0}$ to the distribution of an inversion statistic on ordered set partitions. This verifies one part of the Rise Version of the Delta Conjecture \cite{delta-conjecture} in which we take an inner product with $p_{1^n}$ and set $t=0$ or $q=0$.

We define $\osp{n}{k}$ to be the collection of ordered partitions of $\{1,2,\ldots,n\}$ into exactly $k$ blocks. Given a subset $S$ of $ \{1,2,\ldots,n\}$, we set $\osp{n}{S}$ to be the ordered set partitions in which the minimal elements of the blocks are exactly the elements of $S$. For example, $7|236|45|1$ is an element of $\osp{7}{\{1,2,4,7\}}$ and $\osp{7}{4}$. For any $\hooks \in \{0,1\}^{n}$, we write 
\begin{align}
\set(\hooks) = \{i : \hooks_{i} = 1 \}.
\end{align}
Given an ordered set partition $\pi$, we define a statistic $\inv(\pi)$ that counts the number of pairs of entries $(a,b)$ such that $a>b$, $a$ appears in a block strictly to the left of $b$'s block, and $b$ is minimal in its block. For example, $\inv(5|24|13) = 4$.

\begin{cor}
\label{cor:t=0}
\begin{align*}
\tes{\hooks}{q}{0}  &= \sum_{\pi \in \osp{n}{\set(\hooks)}} q^{\inv(\pi)} = \sum_{j=1}^{n} \qint{\hooks_1+ \ldots + \hooks_j}
\end{align*}
where the last equality uses \cite{rw}. Furthermore, summing over all $\hooks \in \{0,1\}^n$ with $|\hooks| = k+1 \geq 1$ yields
\begin{align*}
\left. \hilb \Delta^{\prime}_{e_{k}} e_n \right|_{t=0} &= \qint{k+1}! \qstir{n}{k+1}{q}. 
\end{align*}
where $\qstir{n}{k}{q}$ is defined by the recursion
\begin{align*}
\qstir{n}{k}{q} = \qstir{n-1}{k-1}{q} + \qint{k} \qstir{n-1}{k}{q} 
\end{align*}
with initial conditions $\qstir{0}{0}{q} = 1$ and $\qstir{n}{k}{q}= 0$ if $k < 0$ or $n < k$.
\end{cor}

In order to prove Corollary \ref{cor:t=0}, we first note that Levande defined a map $\tesset{1^n} \to \S_n$ \cite{levande}. We will denote this map by $L_n$. Furthermore, Levande used a weight-preserving, sign-reversing involution to prove
\begin{align}
\label{levande1}
\sum_{\substack{U \in \tesset{1^n} \\ L_n(U) = \sg}} \wt(U;q,0)  &= q^{\inv(\sg)}.
\end{align}
for any $\sg \in \S_n$. Summing \eqref{levande1} over all permutations $\sg \in \S_n$ yields
\begin{align}
\tes{1^n}{q}{0} &= \qint{n}!.
\end{align}

We extend Levande's results to our setting as follows. For any $\hooks \in \{0,1\}^{n}$, we define a map $L_{\hooks}: \tesset{\hooks} \to \osp{n}{\set(\hooks)}$. To define $L_{\hooks}$, we first map a Tesler matrix $U$ with hook sums $\hooks$ to an intermediary array. This array is created as follows:
\begin{enumerate}
\item First, read the entries of the diagonal $U_{j,j}$ for $j = n$ to 1. If $U_{j,j}> 0$, record a $j$ in the rightmost column of the array $U_{j,j}$ times, recording from top to bottom. After this step, the array will have a single column of length $|\hooks|$ which weakly decreases from top to bottom.
\item For every $j = n$ to $1$, read up the $j$th column from $U_{j-1,j}$ to $U_{1,j}$. For every $U_{i,j} >0$, find the highest $j$ in the array that currently has no entries to its left. Place an $i$ to its left. Place an $i$ in this manner $U_{i,j}$ times.
\end{enumerate}
For example, we send
\begin{displaymath}
\left[ \begin{array}{cccc}
0 & 0 & 0 & 1 \\
0 & 0 & 1 & 0 \\
0 & 0 & 1 & 0 \\
0 & 0 & 0 & 2 \end{array} \right] \mapsto
\left( \begin{array}{cc}
1 & 4  \\
\phantom{1} & 4  \\
2 & 3 \end{array} \right) .
\end{displaymath}
Given such an array, we produce an ordered set partition by the following process.
\begin{enumerate}
\item Read the leftmost entries in each row of the array from bottom to top. Make these the minimal elements in $k$ different blocks, from left to right.
\item For each $i = 1$ to $n$ which is not yet placed into the ordered set partition, find the lowest row in the array in which it appears. Place it in the block which contains the leftmost entry of that row.
\end{enumerate}
Continuing our example, we obtain the ordered set partition $23|4|1$.

\begin{lemma}
\label{lemma:L-beta}
For any $\hooks \in \{0,1\}^{n}$, $L_{\hooks}$ is well-defined. Furthermore, for any $\pi \in \osp{n}{\set(\hooks)}$, 
\begin{align}
 \sum_{\substack{U \in \tesset{\hooks} \\ L_{\hooks}(U) = \pi}} \wt(U; q, 0) &= q^{\inv(\pi)}  .
\end{align}
\end{lemma}

Once we have proved Lemma \ref{lemma:L-beta}, we simply sum over all $\pi \in \osp{n}{\set{\hooks}}$ to prove Corollary \ref{cor:t=0}.

\begin{proof}[Proof of Lemma \ref{lemma:L-beta}]
First, we argue that an array can always be created from $U$ in the manner described above. Consider the $j$th column of $U$. At this point in the process of creating the array associated with $U$, we have processed all entries of the form $U_{j,k}$ for any $k \geq j$. Since the $j$th hook sum of $U$ is nonnegative, there are enough $j$'s available for us to place $U_{i,j}$ $i$'s to the left of a $j$ for each $i$.  

Next, we show that we can always create an ordered set partition $\pi \in \osp{n}{\set(\hooks)}$ from such an array. The number of leftmost elements $j$ in the rows of the array is equal to the $j$th hook sum of $U$. Since $\hooks \in \{0,1\}^{n}$, the minimal elements of $\pi$ are unique and they are indeed equal to $\set(\hooks)$. This proves $\pi \in \osp{n}{\set(\hooks)}$. 

Now that $L_{\hooks}$ is well-defined, we wish to create an involution that concludes the proof of the lemma. Our proof is quite similar to the proof for the $\hooks = 1^{n}$ case in \cite{levande}. We begin by noting that, at $t=0$, the weight of $U$ is
\begin{align}
(q-1)^{\entries(U) - \rows(U)} \prod_{U_{i,j} \neq 0} q^{U_{i,j}-1}
\end{align}
since $U$ may not have any negative entries. We will assign a weight to the associated array in a way that corresponds to this weight. This weight will take the form of an array of the same shape as the associated array, except it will be filled with entries $q$, $1$, or $-1$. For every entry $a$ in the array, let $b$ be the entry directly to its right. (If $a$ is in the rightmost column of the array, set $b=n+1$.) If this is not the lowest appearance of the adjacent pair $ab$ in the array, we assign a weight of $q$. If this is the lowest appearance of $ab$ but it is not the lowest appearance of $a$, we assign a weight of $q$ or $-1$. Otherwise, we assign the weight 1. Then we define the weight of the array to be the product of these individual weights. For example, the only way to assign weights to the above array is 
\begin{align*}
\left( \begin{array}{cc}
1 & 4  \\
\phantom{1} & 4  \\
2 & 3 \end{array} \right)  \mapsto 
\left( \begin{array}{cc}
1 & q  \\
\phantom{1} & 1  \\
1 & 1 \end{array} \right) 
\end{align*}
where the weights are in parentheses.  The total weight of this assignment is $q$.

Now, we wish to define an involution $\Phi_{\pi}$ on these weighted arrays. Let $c$ be the highest leading (i.e.\ leftmost in its row) entry such that there exists a $d>c$ with the property that either
\begin{itemize}
\item $d$ appears below $c$'s row but does not appear in $c$'s row, or
\item $d$ appears in $c$'s row and $d$ has a weight of $-1$.
\end{itemize}
Choose $d$ to be the smallest (and therefore leftmost) entry in $c$'s row that satisfies one of these conditions. In the first case, $\Phi_{\pi}$ inserts a $d$ into $c$'s row along with a $-1$ weight. In the second case, $\Phi_{\pi}$ removes the $d$ from $c$'s row along with its $-1$ weight. If no such $c$ and $d$ exist, $\Phi_{\pi}$ leaves the array as a fixed point. 

To prove that $\Phi_{\pi}$ is an involution, first consider the case where $d$ appears below $c$'s row but does not appear in $c$'s row. Set $c^{\prime}$ (respectively $d^{\prime}$) to be the largest (resp.\ smallest) element in $c$'s row that is less than (resp.\ greater than) $d$; in other words, if $d$ were in $c$'s row $c^{\prime}$ and $d^{\prime}$ would be its left and right neighbors, respectively. ($d^{\prime}$ may be empty, in which case we consider it to be $n+1$.) We need to argue that $d$ can be inserted between $c^{\prime}$ and $d^{\prime}$ with a weight of -1, which can only happen if the lowest appearance of the successive pair $d d^{\prime}$ is in $c$'s row. Say that there is some lower occurrence of $d d^{\prime}$. This must occur below $c$'s row, which contains the adjacent pair $c^{\prime} d^{\prime}$. This cannot happen, by the way in which we create these arrays. By a similar argument, the resulting array is valid, i.e.\ the new adjacent pairs $c^{\prime}d$ and $d d^{\prime}$ obey the defining property of our matrices: if $ai$ and $bi$ are pairs with $ai$ occurring above $bi$, then $a > b$. 
%seems to work, but kind of tedious

Now assume that $d$ appears in $c$'s row and $d$ has a weight of $-1$. Removing $d$ creates the adjacent pair $c^{\prime} d^{\prime}$. We need to show that any weight that had been assigned to $c^{\prime}$ is still valid now that its right neighbor is $d^{\prime}$ instead of $d$. If $c^{\prime}$ had been assigned a weight of $q$, then there must have been a lower occurrence of $c^{\prime}$, which must still exist. If $c^{\prime}$ had been assigned a 1, then we must be considering the lowest appearance of $c^{\prime}$, and removing $d$ does not alter this. Finally, $c^{\prime}$ cannot be weighted with a $-1$ by the minimality of $d$. 
%validity?

It is clear that $\Phi_{\pi}$ is an involution that reverses signs of its non-fixed points. It only remains to investigate the fixed points of $\Phi_{\pi}$. In such a fixed point, for every $a$ that is the leftmost entry in a row of the array and $b>a$, $b$ must appear in $a$'s row. Furthermore, the weight associated with $b$ must either be $q$ or 1 (which can only occur if the entry immediately to the left of $b$ contains the lowest appearance of that element). Fixed points also contain no $-1$ weights. From this, we can see that each $\pi$ has a unique fixed point. It is the array created by the following process:
\begin{enumerate}
\item Write the blocks of $\pi$ as rows in the array, from left to right in $\pi$ and bottom to top in the array.
\item For each entry in the array $b$ and each leftmost element $a$ that appears above $b$, add $b$ to $a$'s row if $a<b$. 
\item Add a weight of $q$ at all possible positions.
\end{enumerate}
For example, the fixed point associated with $\pi = 23|4|1$ is 
\begin{align*}
\left( \begin{array}{cccc}
1 & 2 & 3 & 4 \\
& & & 4 \\
& & 2 & 3 \end{array} \right) 
\ \ \ 
\left( \begin{array}{cccc}
1 & 1 & 1 & 1 \\
& & & 1 \\
&  & q & q \end{array} \right) .
\end{align*}
The weight of this fixed point is equal to the number of minimal elements to the right of $b$ in $\pi$ that are less than $b$ for each $b$. This is exactly the $\inv$ statistic on ordered set partitions.
\end{proof}

\section{Combinatorics at $t=1$}
\label{sec:t=1}

Recall that the weight of an $n \times n$ Tesler matrix $U$ is given by
\begin{align*}
\wt(U;q,t) &= (-1)^{\pos(U) - \posrows(U)} M^{\nonzero(U) - n} \prod_{U_{i,j} \neq 0} \qtint{U_{i,j}} 
\end{align*}
where $M = (1-q)(1-t)$. At $t=1$, it follows that this weight is zero unless $\nonzero(U) = n$, which occurs if and only if every row contains exactly one nonzero entry. 

Following the vocabulary of \cite{tesler-lots}, we say that a Tesler matrix $U \in \tesset{\hooks}$ is \emph{permutational} if it has exactly one nonzero entry in each row, and we denote the set of these matrices by $\ptesset{\hooks}$. Then
\begin{align}
\tes{\hooks}{q}{1} &= \sum_{U \in \ptesset{\hooks}} \wt(U; q,1) .
\end{align}
In general, the collection of permutational Tesler matrices is much easier to manipulate than the set of all Tesler matrices. In particular, we will see that there is a bijection between permutational Tesler matrices and a certain class of ordered set partitions. This bijection will allow us to give several nice formulas for $\tes{\hooks}{q}{1}$ and $\tes{\hooks}{1}{1}$.

\subsection{Permutational Tesler matrices}
\label{ssec:ptes}

For $\hooks \in \mathbb{Z}^n$, we define $\set(\hooks) = \{i : \hooks_i \neq 0\}$. Recall that, for any set $S$, $\osp{n}{S}$ is the set of ordered set partitions $\pi$ such that the set of minimal entries in the blocks of $\pi$ is equal to $S$. If we number the blocks of $\pi$ as $\pi_1, \pi_2, \ldots$ from left to right, it will be helpful to use the notation $\bl_{i}(\pi)$ to mean the number of the block containing $i$. 

Given $\hooks \in \mathbb{N}^n$, we will define a bijection $\psi_{\hooks} : \osp{n}{\set(\hooks)} \to \ptesset{\hooks}$. We will then extend this idea to any $\hooks \in \mathbb{Z}^n$. In order to define this bijection, we first define a vector $\target(\pi)$ by $\target_i(\pi) = j$ where $j > i$, $\bl_j(\pi) \geq \bl_i(\pi)$,  and $\bl_j(\pi)$ is minimal among all such choices of $j$; if no such $j$ exist, then $\target_i(\pi) = i$. In words, we begin at $i$ in $\pi$ (written with its blocks in increasing order) and move to the right; $j$ is the first element we encounter that is greater than $i$. If we never encounter a larger element, $i$ is its own target. For example, if $\pi = 3|12|4$, then $\target(\pi) = (2,4,4,4)$.

Given $\hooks \in \mathbb{N}^n$ and $\pi \in \osp{n}{\set(\hooks)}$ we define a second vector $\tail(\hooks, \pi)$ by
\begin{align}
\tail_{i}(\hooks, \pi) &= \sum_{r=m_i(\pi)}^{\bl_{i}(\pi)} \hooks_{\min  \pi_r  }
\end{align}
where $m_i(\pi)$ is minimal such that $m_i(\pi) \leq \bl_i(\pi)$ and $\pi_{m_i(\pi)-1}$ contains an entry larger than $i$. (If no block to the left of $\pi_{\bl_i(\pi)}$ contains an entry larger than $i$, then $m_i(\pi)=1$.) In words, we begin at $i$ in $\pi$ and scan to the left until we find an entry larger than $i$. The block to the right of the block containing this larger entry is block $m_i(\pi)$. Then we sum the entries of $\hooks$ indexed by the minimal elements of the blocks $\pi_{m_i(\pi)}, \pi_{m_i(\pi)+1}, \ldots, \pi_{\bl_i(\pi)}$. For example, if $\hooks = (2,0,3,1)$ and $\pi = 3|12|4$, 
\begin{align}
\tail_1(\hooks, \pi) &= \hooks_1 = 2 \\
\tail_2(\hooks, \pi) &= \hooks_1 = 2 \\
\tail_3(\hooks, \pi) &= \hooks_3 = 3 \\
\tail_4(\hooks, \pi) &= \hooks_3 + \hooks_1 + \hooks_4 = 3 + 2 + 1 = 6.
\end{align}

Now we are ready to define a bijection
\begin{align}
\psi_{\hooks} : \osp{n}{\set(\hooks)} \to \ptesset{\hooks}
\end{align}
given $\hooks \in \mathbb{N}^n$. Now that we have defined the target and tail vectors, the bijection is not difficult to state. Given $\pi \in \osp{n}{\set(\hooks)}$, we form a permutational Tesler matrix by placing $\tail_i(\hooks, \pi)$ in row $i$ and column $\target_i(\pi)$. All other entries are zero. Continuing the example above,
\begin{align}
\psi_{(2,0,3,1)}(3|12|4) &= \left[ \begin{array}{cccc}
0 & 2 & 0 & 0 \\
0 & 0 & 0 & 2 \\
0 & 0 & 0 & 3 \\
0 & 0 & 0 & 6
 \end{array} \right] .
\end{align} 

\begin{prop}
\label{prop:ptes-bijection}
The map $\psi_{\hooks}$ is a bijection from $\osp{n}{\set(\hooks)}$ to $\ptesset{\hooks}$ for $\hooks \in \mathbb{N}^n$. Furthemore, for any $\pi \in \osp{n}{\set(\hooks)}$ we have 
\begin{align*}
\wt(\psi_{\hooks}(\pi); q, 1) = \prod_{i=1}^{n} \qint{\tail_i(\hooks, \pi)} .
\end{align*}
This implies
\begin{align*}
\tes{\hooks}{q}{1} &= \sum_{\pi \in \osp{n}{\set(\hooks)}} \prod_{i=1}^{n} \qint{\tail_i(\hooks, \pi)} .
\end{align*}
\end{prop}

\begin{proof}
The crux of the proof is verifying that $\psi_{\hooks}$ is a bijection. The final two statements follow directly from the definitions.

First, we need to verify that the range of $\psi_{\hooks}$ is actually contained in $\ptesset{\hooks}$. Given $\pi \in \osp{n}{\set(\hooks)}$, it is clear that $U = \psi_{\hooks}(\pi)$ has at most one nonzero entry per row, namely the value $U_{i,\target_i(\pi)} = \tail_i(\hooks, \pi)$ in each row $i=1,2,\ldots,n$. Furthermore, every $\tail_i(\hooks, \pi)$ is a nonempty sum of the strictly positive entries in $\hooks$, so $U_{i,j}$ must be positive and $U$ must have exactly one nonzero entry per row. Since $\target_i(\pi) \geq i$ by definition, $U$ is also upper triangular.

Finally, we must show that $\hooksvec(U) = \hooks$. By the definition of $\psi_{\hooks}$, we have
\begin{align*}
\hooksvec_i(U) &= \tail_i(\hooks, \pi) - \sum_{c < i : \target_c(\pi) = i} \tail_c(\hooks, \pi).
\end{align*}
Recall that $\tail_i(\hooks, \pi)$ is the sum $\sum_{r=m_i(\pi)}^{\bl_i(\pi)} \hooks_{\min \pi_r  }$ where $\pi_{m_i(\pi)}$ is the leftmost block of $\pi$ such that every element in the blocks $\pi_{m_i(\pi)}, \pi_{m_i(\pi)+1}, \ldots, \pi_{i-1}$ is less than $i$. 

Let $c_1, c_2, \ldots, c_k$ be the set of elements less than $i$ such that $\target_{c_j}(\pi) = i$ as they occur from left to right in $\pi$. For example, if $\pi = 8|25|6|14|379$ and $i=7$, then $m_i(\pi) = 2$, $c_1 = 6$, $c_2=4$, and $c_3=3$. For each $j$, every element between $c_j$ and $i$ in $\pi$ must be less than $i$, so $c_1, c_2, \ldots, c_k$ appear among the blocks $\pi_{m_i(\pi)}, \ldots, \pi_{\bl_i(\pi)}$ and we must have $c_1 > c_2 > \ldots > c_k$. Furthermore, each entry between $c_j$ and $c_{j+1}$ must be less than $c_{j+1}$; to see this, we note that if any such element were larger than $c_{j+1}$ and its target were not equal to $i$, then $c_{j+1}$'s target could not be $i$, which is a contradiction. It follows that $m_{c_j}(\pi) = \bl_{c_{j-1}}(\pi)$ for each $j \geq 2$ and $m_{c_1}(\pi) = m_i(\pi)$. By the definition of the tail vector, this implies
\begin{align}
\tail_{c_j}(\hooks, \pi) &= \sum_{r=m_{c_j}(\pi)}^{\bl_{c_j}(\pi)} \hooks_{\min \pi_r} = \sum_{\bl_{c_{j-1}}(\pi)}^{\bl_{c_j}(\pi)} \hooks_{\min \pi_r}
\end{align}
for $j \geq 2$ and
\begin{align}
\tail_{c_1}(\hooks, \pi) &= \sum_{r=\bl_{m_i(\pi)}(\pi)}^{\bl_{c_1}(\pi)} \hooks_{\min \pi_r}.
\end{align}

Now we note that we can rewrite $\hooksvec_i(U)$ as
\begin{align}
\hooksvec_i(U) &= \tail_i(\hooks, \pi) - \sum_{c < i : \target_c(\pi) = i} \tail_c(\hooks, \pi) \\
&= \tail_i(\hooks, \pi) - \sum_{j=1}^{k} \tail_{c_j}(\hooks, \pi) \\
&= \sum_{r=m_i(\pi)}^{\bl_i(\pi)} \hooks_{\min \pi_r} - \sum_{j=1}^{k} \sum_{r=m_{c_j}(\pi)}^{\bl_{c_j}(\pi)} \hooks_{\min \pi_r} \\
&= \sum_{r=m_i(\pi)}^{\bl_i(\pi)} \hooks_{\min \pi_r} - \sum_{r=\bl_{m_{i}(\pi)}(\pi)}^{\bl_{c_1}(\pi)} \hooks_{\min \pi_r} - \sum_{j=2}^{k} \sum_{r=\bl_{c_{j-1}}(\pi)}^{\bl_{c_{j}}(\pi)} \hooks_{\min \pi_r} .
\end{align}
Since the $c_j$'s decrease from left to right, each $\bl_{c_j}(\pi)$ must be distinct. Therefore we can rewrite the previous identity as
\begin{align} 
\label{tails-eqn-final}
&= \sum_{r=m_i(\pi)}^{\bl_i(\pi)} \hooks_{\min \pi_r} - \sum_{r=m_i(\pi)}^{\bl_{c_k}(\pi)} \hooks_{\min \pi_r} .
\end{align}

The only question that remains is the value of $\bl_{c_k}(\pi)$. If $i$ is not minimal in its block, then it must be the target of the element immediately to its left, which must share a block with $i$. In this case, this element immediately to $i$'s left must be $c_k$, so $\bl_{c_k}(\pi) = \bl_i(\pi)$ and \eqref{tails-eqn-final} must be equal to zero. If $i$ is minimal in its block, then the element immediately to its left is still $c_k$, but in this case we have $\bl_{c_k} = \bl_i(\pi)-1$, which means that \eqref{tails-eqn-final} is $\hooks_{\min \pi_{\bl_i(\pi)}} = \hooks_i$. Therefore $\hooksvec(U) = \hooks$, as desired.
\end{proof}

\begin{prop}
\label{prop:ptes-weight-neg}
Now we allow $\hooks \in \mathbb{Z}^n$ and take $\pi \in \osp{n}{\set(\hooks)}$. We extend the definition of $\tail(\hooks, \pi)$ to this new setting with no changes. Then the final formula in Proposition \ref{prop:ptes-bijection} still holds, i.e.\ 
\begin{align*}
\tes{\hooks}{q}{1} &= \sum_{\pi \in \osp{n}{\set(\hooks)}} \prod_{i=1}^{n} \qint{\tail_i(\hooks, \pi)} .
\end{align*}
\end{prop}

\begin{proof}
We claim that the only place in the proof of Proposition \ref{prop:ptes-weight-neg} that relies on $\hooks \in \mathbb{N}^n$ is the assertion that $\tail_i(\hooks, \pi)$ is always positive. When $\hooks \in \mathbb{Z}^n$, $\tail_i(\hooks, \pi)$ can be negative (which is no problem, since we allow negative entries in our Tesler matrices) or zero; this is a problem, since Tesler matrices cannot have a row containing only zeros. For example, if $\hooks = (2,0,-3,1)$ and $\pi = 3|12|4$, 
\begin{align}
\tail_4(\hooks, \pi) &= \hooks_3 + \hooks_1 + \hooks_4 = -3 + 2 + 1 = 0
\end{align}
and the bottom row of $\phi_{\hooks}(\pi)$ contains only zeros. Therefore, for $\hooks \in \mathbb{Z}^n$, it is possible that $\psi_{\hooks}(\pi)$ may produce a matrix that is not in $\ptesset{\hooks}$, so $\psi_{\hooks}$ cannot be a bijection between $\osp{n}{\set(\hooks)}$ and $\ptesset{\hooks}$.

However, we note that, if we ever have $\tail_{i}(\hooks, \pi) = 0$, then 
\begin{align*}
\prod_{i=1}^{n} \qint{\tail_{i}(\hooks, \pi)} = 0 .
\end{align*}
In other words, the elements $\pi \in \osp{n}{\set(\hooks)}$ on which $\psi_{\hooks}$ fails to be a bijection do not contribute to the right-hand side of the identity in the proposition. $\psi_{\hooks}$ is, in fact, a bijection between the elements $\pi \in \osp{n}{\set(\hooks)}$ such that $\tail_i(\hooks, \pi) \neq 0$ for every $i$ and $\ptesset{\hooks}$. Furthermore, $\psi_{\hooks}$ satisfies
\begin{align*}
\wt(\psi_{\hooks}(\pi); q, 1) = \prod_{i=1}^{n} \qint{\tail_i(\hooks, \pi)} .
\end{align*}
This completes the proof.
\end{proof}

\subsection{Connections to parking functions}
\label{ssec:pf}

In this subsection, we will see that the product of tails associated to an ordered set partition in Subsection \ref{ssec:ptes} has a connection to certain modifications of parking functions. This connection is especially interesting because of the conjectured relationships between Tesler functions and parking functions, such as the Shuffle Conjecture \cite{shuffle} and the new Delta Conjecture \cite{delta-conjecture}. We focus on the case $\hooks \in \{0,1\}^n$ in this subsection, which will extend the results of \cite{tesler-lots} for $\hooks = 1^n$. 

A \emph{parking function} of order $n$ is a function $f : \{1,2,\ldots,n\} \to \{1,2,\ldots,n\}$ such that the $i$th entry in the increasing rearrangement of $f(1), f(2), \ldots, f(n)$ is at most $i$ for every $i=1$ to $n$. A parking function is often written by simply listing $f(1) f(2) \ldots f(n)$ from left to right. For example, $311$ is a parking function but $313$ is not. 
We will often write $f$ as a sequence in this manner, and we use $f(i)$ and $f_i$ interchangeably. We will also write $\pf_n$ to mean the set of all parking functions of order $n$. Although parking functions were first defined in the setting of computer science \cite{parking}, they have become important objects in algebraic combinatorics due to their connection to areas such as diagonal harmonics \cite{gh-model, shuffle, haglund-book}.

Parking functions inherited their name because they can be interpreted in the setting of a parking lot. Given cars $1,2,\ldots,n$ and a parking lot with spaces numbered $1,2,\ldots,n$, we imagine that the cars enter the lot one at a time in increasing order. Each car $i$ has a preferred spot $f_i$. It drives to that spot; if the spot is still unoccupied, it parks there. Otherwise, it parks in the smallest available spot $j$ such that $j$ is unoccupied and $j > f_i$. The function $f$ is a parking function by the above definition if and only if all cars park successfully, i.e.\ every car always finds some spot $j$ to park in under the above procedure.

Given a parking function $f$, let $\car_i(f)$ be the number of the car that parks in spot $i$ and $\spot_i(f)$ be the spot that ends up containing car $i$. Note that $\car(f) = (\car_1(f), \car_2(f), \ldots, \car_n(f))$ and $\spot(f) = (\spot_1(f), \spot_2(f), \ldots, \spot_n(f))$ are inverses as permutations.
For example, the parking function $f=5121142$ has $\car(f) = 2345167$ and $\spot(f) = 5123467$. 

We will consider a new property of the cars of a parking function. A car $i$ is called \emph{considerate} if $\spot_i(f) \neq f_j$ for all $j$, i.e.\ no car hoped to park in the spot which car $i$ eventually occupies. Note that this implies that $f_i \neq \spot_i(f)$, i.e.\ car $i$ does not end up in its desired spot. A car $i$ which satisfies $f_i \neq \spot_i(f)$ is sometimes called \emph{unlucky}, so every considerate car must be unlucky. Although unlucky cars have appeared in the literature on parking functions, we are not aware of the use of any condition equivalent to being considerate.

We denote the set of considerate cars in $f$ by $\considerate(f)$. Given a set $S$, we will consider the ``decorated'' parking functions
\begin{align}
\consideratepf_{n,S} &= \{f: f \in \pf_n, S \subseteq \considerate(f)\}. 
\end{align}
For example, $\consideratepf_{3, \{2\}} = \{111, 113, 221\}$. $213$ is not in $\consideratepf_{3, \{2\}}$ because 2 is ``lucky'' in that case, and $112$ is not in $\consideratepf_{3, \{2\}}$ because car 3 desires to park in spot 2, which is where car 2 parks. Note that a given parking function can be in $\consideratepf_{n,S}$ for several different sets $S$.

Given an element $f \in \consideratepf_{n,S}$, we define an ordered set partition by taking the underlying permutation $\car(f)$ and adding bars before each car not in $S$ (except for the leftmost entry in $\car(f)$, which cannot be considerate by definition). We denote this ordered set partition $\car(f, S)$. For example, 
\begin{align}
\car(5121142, \{4,7\}) = 2|34|5|1|67
\end{align}
It is not immediately clear that the result is even an ordered set partition; namely, we need to ensure that bars are only placed at ascents in the permutation. Since every car $i$ in $S$ is unlucky, we must have $f_i < \spot_i(f)$. It follows that each spot $f_i, f_i + 1, \ldots, \spot_i(f)-1$ must contain a car $j < i$; otherwise, car $i$ would have parked in that spot. In particular, the car that parks immediately to the left of $i$ must have some label less than $i$.

We also define a statistic on elements of $\consideratepf_{n,S}$:
\begin{align}
\area(f, S) &= \sum_{i=1}^{n} \spot_i(f) - f_i - | \{f_i, f_i + 1, \ldots, \spot_i(f)\} \cap \{\spot_j(f) : j \in S \}| .
\end{align}
In words, this is the total number of spots passed by cars after passing their desired spot, minus those spots that they pass that contain a ``considerate'' car in $S$ (including possibly their own destinations). One can verify that, when $S = \emptyset$, this is the standard area statistic on parking functions \cite{haglund-book}. In general, this new area statistic is equal to the area of the parking function $f$ minus the areas of the columns $\spot_j(f)$ for all $j \in S$ when $f$ is drawn as a labeled Dyck path minus $|S|$. 

For example, consider $f = 5121142 \in \consideratepf_{7, \{4,7\}}$, so $S = \{4,7\}$. We have $\spot(f) = 5123467$ and $\{\spot_j(f) : j \in S\} = \{3,7\}$. We can compute
\begin{align}
\area(f, S) &= (5-5-0) + (1-1-0) + (2-2-0) \\
&+ (3-1-1) + (4-1-1) + (6-4-0) + (7-2-2) \\
&= 0 + 0 + 0 + 1 + 2 + 2 + 3 \\
&= 8.
\end{align}
The labeled Dyck path for $f=5121142$ is drawn in Figure \ref{fig:labeled-dyck}. To compute the area of $f$, we simply count the number of full squares between the Dyck path and the line $y=x$. Recall that $\{\spot_j(f) : j \in S\} = \{3,7\}$. These columns contribute 2 and 0 squares to the area of $f$, respectively. Then
\begin{align}
\area(f,S) &= \area(f) - 2 - 0 - |S| = 12 - 2 - 0 - 2 = 8.
\end{align}
This labeled Dyck path interpretation is important because it means that the forthcoming Proposition \ref{prop:tesler-pf-q} proves the $q=1$ case of the Rise Version of the Delta Conjecture \cite{delta-conjecture} where we take scalar products with $p_{1^n}$. For more information on parking functions as labeled Dyck paths, see \cite{haglund-book, delta-conjecture}.

\begin{figure}
\begin{center}
\begin{tikzpicture}
\draw[very thin] (0,0) -- (3.5,3.5);
\draw[step=0.5cm, gray, very thin] (0, 0) grid (3.5, 3.5);
\draw[blue, very thick] (0,0) -- (0,1.5) -- (0.5,1.5) -- (0.5,2.5) -- (1.5,2.5) -- (1.5,3) -- (2,3) -- (2,3.5) -- (3.5,3.5);
\node at (0.25,0.25) {2};
\node at (0.25,0.75) {4};
\node at (0.25,1.25) {5};
\node at (0.75,1.75) {3};
\node at (0.75,2.25) {7};
\node at (1.75,2.75) {6};
\node at (2.25,3.25) {1};
\end{tikzpicture}
\end{center}
\caption{The labeled Dyck path associated with the parking function $5121142$. The area of this parking function is 12.}
\label{fig:labeled-dyck}
\end{figure}
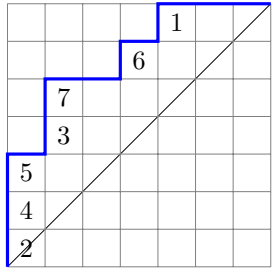

\begin{prop}
\label{prop:tesler-pf-q}
Given $\hooks \in \{0,1\}^n$, let $S = \{1,2,\ldots,n\} \setminus \set(\hooks)$. For any ordered set partition $\pi \in \osp{n}{\set(\hooks)}$,
\begin{align*}
\prod_{i=1}^{n} \qint{\tail_i(\hooks, \pi)} &= \sum_{\substack{f \in \consideratepf_{n, S} \\ \car(f, S) = \pi}} q^{\area(f, S)} .
\end{align*}
\end{prop}

\begin{proof}
For each car $i = 1$ to $n$, we claim that $\tail_i(\hooks, \pi)$ gives the number of possible spots in which car $i$ could have wanted to park, i.e.\ the number of possibilities for $f_i$ for an element $f \in \consideratepf_{n,S}$. For example, if $\hooks = (1,1,1,0,1,1,0)$ and $\pi = 2|34|5|1|67$, we compute
\begin{align*}
\tail(\hooks, \pi) = (1, 1, 2, 2, 3, 5, 5)
\end{align*}
We know that the cars end up in order $2345167$ and that cars $4$ and $7$ are unlucky. Car 1 always gets its preferred spot, so we must have $f_1=5$. Since car 2 is in spot 1, it must have preferred spot 1, so $f_2=1$. Car 3 could have preferred its eventual spot ($f_3=2$), or it could have wanted to park in spot 1, in which case it would have been pushed to spot 2 because car 2 was already parked in spot 1. Therefore car 3 could have wanted to park in 2 different spots. We know car 4 ends up in spot 3, but we also know that it is considerate. Hence, car 4 could not have wanted to park in its eventual spot ($f_4 \neq 3)$. It is possible that $f_4 = 1$ or 2. Car 5 could have wanted to park in spots 1, 2, or 4; it could not have wanted to park in spot 3, because that would contradict the fact that car 4 is considerate. By similar logic, car 6 and car 7 could have wanted to park in 5 different spaces: 1,2,4,5, and 6.

It follows that
\begin{align}
\prod_{i=1}^{n} \tail_i(\hooks, \pi) = |\consideratepf_{n,S}|.
\end{align}
In fact, for any car $i$, we claim
\begin{align}
\qint{\tail_i(\hooks, \pi)} = \sum q^{\spot_i(f) - f_i - | \{f_i, f_i + 1, \ldots, \spot_i(f)\} \cap \{\spot_j(f) : j \in S \}|}
\end{align}
where the right sum is over all potential $f_i$'s that result in car $i$ parking in $\spot_i(f)$. To see this, consider car 5 in the above example again. Car 5 ends up in spot 4, but it could have wanted to park in spots 1, 2, or 4. These possibilities would contribute 2,1, or 0 to $\area(f, S)$, respectively. Summing these contributions as powers of $q$, we obtain $\qint{\tail_5(\hooks, \pi)}$.

\end{proof}

\subsection{The $q=t=1$ case}
\label{ssec:q=t=1}

If we also set $q=1$, we can provide a simple product formula for the Tesler function. This same product formula was proved to hold in the case where $\hooks_i > 0$ for each $i$ in \cite{tesler-lots}.

\begin{prop}
For any $\hooks \in \mathbb{Z}^n$, 
\begin{align*}
\tes{\hooks}{1}{1} &= \sum_{\pi \in \osp{n}{\set(\hooks)}} \prod_{i=1}^{n} \tail_i(\hooks, \pi) \\
&=  \hooks_1(\hooks_1 + n\hooks_2)(\hooks_1 + \hooks_2 + (n-1)\hooks_3) \ldots (\hooks_1 + \hooks_2 + \ldots + \hooks_{n-1} + 2\hooks_n) .
\end{align*}
\end{prop}

\begin{proof}
We will roughly follow the proof of Proposition 6.2 in \cite{tesler-lots}. For any parking function $f$, we set
\begin{align}
\wt_{\hooks}(f) &= \prod_{i=1}^{n} \hooks_{\car_{f_i} (f)} .
\end{align}
In words, for each car $i$ we multiply by the entry in $\hooks$ whose subscript is equal to the number of the car that ends up in $i$'s desired spot, $f_i$. For example, if $\hooks = (2,-1,0,3)$ and $f = 2121$, 
\begin{align}
\wt_{\hooks}(f) &= \hooks_1 \hooks_2 \hooks_1 \hooks_2 = (2)(-1)(2)(-1) = 4.
\end{align}

First, we want to show that, for any $\pi \in \osp{n}{\set(\hooks)}$,
\begin{align}
\prod_{i=1}^{n} \tail_i(\hooks, \pi) &= \sum_{f \in \consideratepf_{n, S}} \wt_{\hooks}(f)
\end{align}
where $S$ is the complement of $\set(\hooks)$. This essentially follows from the proof of Proposition \ref{prop:tesler-pf-q}. For any $f \in \consideratepf_{n,S}$, the spaces where $i$ could have wanted to park, i.e.\ the potential $f_i$, are exactly the spaces which contribute to the sum in the definition of $\tail_i(\hooks, \pi)$.  
%more?

Finally, we claim
\begin{align}
\sum_{f \in \consideratepf_{n, S}} \wt_{\hooks}(f) &= \hooks_1(\hooks_1 + n\hooks_2)(\hooks_1 + \hooks_2 + (n-1)\hooks_3)\\
&\ldots (\hooks_1 + \hooks_2 + \ldots + \hooks_{n-1} + 2\hooks_n) .
\end{align}
This follows from Proposition 6.2 in \cite{tesler-lots}, but we sketch their argument here for the sake of completeness. One well-known way to produce all parking functions in $\pf_n$ is to start with a lot with an extra space, which is labeled $n+1$. Then each car has $n+1$ choices for its preferred space. After the cars are all parked, the numbers of the spaces are ``rotated'' until the unoccupied space is number $n+1$. There are $n+1$ options for rotating and only one of these options has $n+1$ unoccupied, so this corresponds to division by $n+1$. This argument shows that
\begin{align}
|\pf_n| &= (n+1)^n / (n+1) = (n+1)^{n-1}.
\end{align}
Now we run through this procedure again, account for the weight $\wt_{\hooks}(f)$. Car 1 has $n+1$ choices of its preferred spot, yielding a term of $(n+1)\hooks_1$. Car 2 can either try to park where car 1 has just parked, producing a weight of $\hooks_1$, or it can pick one of the $n$ other spots, producing a weight of $n \hooks_2$. We continue this process until car $n$ has 2 ways to prefer a spot which is unoccupied, yielding $2 \hooks_n$, or it can try to park where one of the other $n-1$ cars has already parked. Finally, we rotate the numbers of the spaces until space $n+1$ is empty; this yields a division by $n+1$. One can see that this product is equal to the product in the statement of this proposition.
\end{proof}

\section{Future work}
\label{sec:future}

In this section, we conclude by discussing some ways in which our results could be improved or extended in the future.

In our results so far, we have relied heavily on the fact that we are taking Hilbert series of the various symmetric functions at hand. It is reasonable to ask how Tesler matrices can be used to give formulas for the symmetric functions themselves. For example, in \cite{tesler-frobenius}, Garsia and Haglund use Tesler matrices to give a formula for the symmetric function $\nabla e_n$. A similar (but not equivalent) formula for (rational extensions of) $\nabla e_n$ is given in \cite{gorsky-negut}.

In a similar vein, it would be interesting to obtain symmetric functions whose Hilbert series are equal to $F^{\hooks}_{\mu}$. Such a result would allow us to replace virtual Hilbert series with the actual Hilbert series of these symmetric functions. 

It seems possible that the methods used in Subsection \ref{ssec:pos} could be applied when $e_1$ is replaced by a slightly more complicated function ($e_2$ or $m_2$, for example). Similarly, we may be able to extend the results in Section \ref{sec:t=0} to $\hooks$ with entries not equal to 0 or 1. The computations will be more difficult in these cases, but they may still be tractable. 

As for special cases of Tesler functions, one would hope to be able to find a formula for the function $\tes{\hooks}{q}{0}$ and to derive a parking function interpretation for $\tes{\hooks}{q}{1}$ for any $\hooks \in \mathbb{Z}^n$. It would also be helpful to come up with a parking function interpretation for the full function $\tes{\hooks}{q}{t}$, even just for $\hooks \in \{0,1\}^n$. 

\bibliography{../statistics/statistics}

\newcommand{\etalchar}[1]{$^{#1}$}
\begin{thebibliography}{HHL{\etalchar{+}}05b}

\bibitem[AGR{\etalchar{+}}12]{tesler-lots}
D.~Armstrong, A.~Garsia, B.~Rhoades, J.~Haglund, and B.~Sagan.
\newblock Combinatorics of {Tesler} matrices in the theory of parking functions
  and diagonal harmonics.
\newblock {\em J. Combin.}, 3:451--494, 2012.

\bibitem[BGHT99]{bght-positivity}
F.~Bergeron, A.~M. Garsia, M.~Haiman, and G.~Tesler.
\newblock Identities and positivity conjectures for some remarkable operators
  in the theory of symmetric functions.
\newblock {\em Meths. and Appls. of Analysis}, 6(3):363--420, 1999.

\bibitem[GH93]{gh-model}
A.~M. Garsia and M.~Haiman.
\newblock A graded repesentation model for the {Macdonald} polynomials.
\newblock {\em Proc. Nat. Acad. Sci.}, 90:3607--3610, 1993.

\bibitem[GH03]{qtcatalan}
A.~Garsia and J.~Haglund.
\newblock A proof of the $q,t$-{Catalan} positivity conjecture.
\newblock {\em Adv. Math}, 175:319--334, 2003.

\bibitem[GH14]{tesler-frobenius}
A.~Garsia and J.~Haglund.
\newblock A polynomial expression for the character of diagonal harmonics.
\newblock To appear in \emph{Ann. of Combin.}, 2014.

\bibitem[GHX14]{tesler-constant}
A.~Garsia, J.~Haglund, and G.~Xin.
\newblock Constant term methods in the theory of {Tesler} matrices and
  {Macdonald} polynomial operators.
\newblock {\em Ann. Combin.}, 38:83--109, 2014.

\bibitem[GHXZ14]{macdonald-pieri}
A.~M. Garsia, J.~Haglund, G.~Xin, and M.~Zabrocki.
\newblock Some new applications of the {Stanley}-{Macdonald} {Pieri} rules.
\newblock arXiv:1407.7916, July 2014.

\bibitem[GN13]{gorsky-negut}
E.~Gorsky and A.~Negut.
\newblock Refined knot invariants and {Hilbert} schemes.
\newblock arXiv:1304.3328, April 2013.

\bibitem[Hag08]{haglund-book}
J.~Haglund.
\newblock {\em The $q,t$-{Catalan} Numbers and the Space of Diagonal
  Harmonics}.
\newblock Amer. Math. Soc., 2008.
\newblock Vol. 41 of University Lecture Series.

\bibitem[Hag11]{tesler-hilbert}
J.~Haglund.
\newblock A polynomial expression for the {Hilbert} series of the quotient ring
  of diagonal coinvariants.
\newblock {\em Adv. in Math.}, 227:2092--2106, 2011.

\bibitem[Hai01]{haiman-positivity}
M.~Haiman.
\newblock Hilbert schemes, polygraphs, and the {Macdonald} positivity
  conjecture.
\newblock {\em J. Amer. Math. Soc.}, 14:941--1006, 2001.

\bibitem[Hai02]{delta}
M.~Haiman.
\newblock Vanishing theorems and character formulas for the {Hilbert} scheme of
  points in the plane.
\newblock {\em Invent. Math.}, 149:371--407, 2002.

\bibitem[HHL05a]{hhl}
J.~Haglund, M.~Haiman, and N.~Loehr.
\newblock A combinatorial formula for {Macdonald} polynomials.
\newblock {\em J. Amer. Math. Soc.}, 18:735--761, 2005.

\bibitem[HHL{\etalchar{+}}05b]{shuffle}
J.~Haglund, M.~Haiman, N.~Loehr, J.~B. Remmel, and A.~Ulyanov.
\newblock A combinatorial formula for the character of the diagonal
  coinvariants.
\newblock {\em Duke Math. J.}, 126:195--232, 2005.

\bibitem[HRW15]{delta-conjecture}
J.~Haglund, J.~B. Remmel, and A.~T. Wilson.
\newblock {The Delta Conjecture}.
\newblock arXiv:1509.07058, September 2015.

\bibitem[KW66]{parking}
A.~G. Konheim and B.~Weiss.
\newblock An occupancy discipline and applications.
\newblock {\em SIAM J. Applied Math.}, 14:17--76, 1966.

\bibitem[Lev12]{levande}
P.~Levande.
\newblock {\em Combinatorial Structures and Generating Functions of {Fishburn}
  Numbers, Parking Functions, and {Tesler} Matrices}.
\newblock PhD thesis, U. of Penn., 2012.

\bibitem[Mac95]{macdonald}
I.~Macdonald.
\newblock {\em Symmetric Functions and Hall Polynomials}.
\newblock Oxford University Press, second edition, 1995.

\bibitem[MMR14]{tesler-polytope}
K.~M\'esz\'aros, A.~H. Morales, and B.~Rhoades.
\newblock The polytope of {Tesler} matrices.
\newblock arXiv:1409.8566, September 2014.

\bibitem[RW15]{rw}
J.~B. Remmel and A.~T. Wilson.
\newblock An extension of {MacMahon's} equidistribution theorem to ordered set
  partitions.
\newblock {\em J. Combin. Theory, Ser. A}, 134:242--277, August 2015.

\bibitem[Sta99]{ec2}
R.~P. Stanley.
\newblock {\em Enumerative Combinatorics}, volume~2.
\newblock Cambridge University Press, 1999.

\end{thebibliography}
\bibliographystyle{alpha}

\end{document}